\documentclass[11pt]{article}
\usepackage{graphicx}
\usepackage[top=1in,bottom=1in,left=1in,right=1in]{geometry}
\usepackage[sort&compress]{natbib} 
\usepackage{amsfonts, times, amsmath, amssymb, amsthm, constants, bbm}
\usepackage{MnSymbol}
\usepackage{mathtools,caption}
\usepackage[mathscr]{eucal}
\usepackage[noline,ruled]{algorithm2e}

\usepackage{color,times,float}
\usepackage[inline]{enumitem}
\setlist[enumerate,1]{label=\textit{\alph*)}}
\definecolor{darkred}{RGB}{100,0,0}
\definecolor{darkgreen}{RGB}{0,100,0}
\definecolor{darkblue}{RGB}{0,0,150}
\newenvironment{CompactEnumerate}{
\begin{list}{\arabic{enumi}.}{%
\usecounter{enumi}
\setlength{\leftmargin}{4pt}
\setlength{\itemindent}{1  pt}
\setlength{\topsep}{-1pt}
\setlength{\itemsep}{-2pt}
}}
{\end{list}}
\usepackage{hyperref}
\hypersetup{colorlinks=true, linkcolor=darkred, citecolor=darkgreen, urlcolor=darkblue}

\usepackage{url}
\usepackage{enumitem}
\usepackage[nameinlink]{cleveref}

\def\fdr{\text{FDR}}
\def\fdp{\text{FDP}}
\def\tdp{\text{TDP}}
\def\tdr{\text{TDR}}

\newtheorem{thm}{Theorem}
\newtheorem{prp}{Proposition}
\newtheorem{lem}{Lemma}

\newtheorem{remark}{Remark}
\newtheorem{Def}{Definition}
\newtheorem{rem}{Remark}
\newtheorem{example}{Example}

\def\beq{\begin{equation*}} 
\def\eeq{\end{equation*}}
\def\beqn{\begin{eqnarray*}}
\def\eeqn{\end{eqnarray*}}
\def\Bitem{\begin{itemize}\setlength{\itemsep}{.2in}}
\def\bitem{\begin{itemize}\setlength{\itemsep}{.05in}}
\def\eitem{\end{itemize}}
\def\Benum{\begin{enumerate}\setlength{\itemsep}{.2in}}
\def\benum{\begin{enumerate}\setlength{\itemsep}{.05in}}
\def\eenum{\end{enumerate}}
\def\bmult{\begin{multline*}}
\def\emult{\end{multline*}}
\def\bcenter{\begin{center}}
\def\ecenter{\end{center}}
\def\bframe{\begin{frame}}
\def\eframe{\end{frame}}

\newcommand{\thmref}[1]{Theorem~\ref{thm:#1}}
\newcommand{\prpref}[1]{Proposition~\ref{prp:#1}}

\newcommand{\lemref}[1]{Lemma~\ref{lem:#1}}
\newcommand{\secref}[1]{Section~\ref{sec:#1}}
\newcommand{\figref}[1]{Figure~\ref{fig:#1}}

\def\cC{\mathcal{C}}

\def\cF{\mathcal{F}}
\def\cG{\mathcal{G}}
\def\cH{\mathcal{H}}

\def\cL{\mathcal{L}}

\def\cN{\mathcal{N}}

\def\cR{\mathcal{R}}

\def\cX{\mathcal{X}}

\def\DR{\mathrm{EDR}}

\def\bP{\mathbf{P}}

\def\bR{\mathbf{R}}

\def\bX{\mathbf{X}}

\newcommand{\E}{\operatorname{\mathbb{E}}}

\def\eps{\varepsilon}

\def\1{\mathbbm{1}}

\newcommand{\CGAI}{Contextual Generalized Alpha-Investing }
\newcommand{\CwGAI}{Context-weighted Generalized Alpha-Investing }
\newcommand{\LORD}{Level based On Recent Discovery }

\newcommand{\CwLORDplus}{context-weighted LORD++ }

\usepackage{caption}
\usepackage{subcaption}
\definecolor{purple}{rgb}{0.4,.1,.9}
\definecolor{new}{rgb}{0.5,0.1,0.1}

\begin{document}

\thispagestyle{empty}

\title{Contextual Online False Discovery Rate Control}
 \author{Shiyun Chen\thanks{ Department of Mathematics, University of California, San Diego. \href{mailto:shc176@ucsd.edu}{shc176@ucsd.edu}. Part of the work done with the author was an intern at Amazon.}\and Shiva Kasiviswanathan\thanks{Amazon, Sunnyvale, USA. \href{mailto:kasivisw@gmail.com}{kasivisw@gmail.com}.}} 
\date{}
\maketitle
\begin{abstract}
Multiple hypothesis testing, a situation when we wish to consider many hypotheses, is a core problem in statistical inference that arises in almost every scientific field. In this setting, controlling the false discovery rate (FDR), which is the expected proportion of type I error, is an important challenge for making meaningful inferences. In this paper, we consider the problem of controlling FDR in an online manner. Concretely, we consider an ordered, possibly infinite, sequence of hypotheses, arriving one at each timestep, and for each hypothesis we observe a p-value along with a set of features specific to that hypothesis. The decision whether or not to reject the current hypothesis must be made immediately at each timestep, before the next hypothesis is observed. The model of multi-dimensional feature set provides a very general way of leveraging the auxiliary information in the data which helps in maximizing the number of discoveries.

We propose a new class of powerful online testing procedures, where the rejections thresholds (significance levels) are learnt sequentially by incorporating contextual information and previous results. We prove that any rule in this class controls online FDR under some standard assumptions. We then focus on a subclass of these procedures, based on weighting significance levels, to derive a practical algorithm that learns a parametric weight function in an online fashion to gain more discoveries. We also theoretically prove, in a stylized setting, that our proposed procedures would lead to an increase in the achieved statistical power over a popular online testing procedure proposed by~\citet{javanmard2018online}. Finally, we demonstrate the favorable performance of our procedure, by comparing it to state-of-the-art online multiple testing procedures, on both synthetic data and real data generated from different applications.
\end{abstract}

\newpage
\tableofcontents
\newpage

\section{Introduction} \label{sec:intro}
Multiple hypotheses testing - controlling overall error rates when performing simultaneous hypothesis tests - is a well-established area in statistics with applications in a variety of scientific disciplines~\citep{dudoit2007multiple, dickhaus2014simultaneous, roquain2011type}. This problem has become even more important with modern data science, where standard data pipelines involve performing a large number of hypotheses tests on complex datasets, e.g., does this change to my webpage improve my click-through rate, is this ad effective for this population of users, or is this gene mutation associated with certain trait? 

In hypothesis testing problems, each hypothesis is summarized to one p-value, which we use to decide whether to reject the null hypothesis, i.e., claim it as a non-null. Typically, a hypothesis is rejected if p-value is below some significance level. The rejected hypotheses are called {\em discoveries}, and the subset of these that were truly null but mistakenly rejected are called {\em false discoveries}. The {\em false discovery rate} (FDR) namely, the expected fraction of discoveries that are false positives is the criterion of choice for statistical inference in multiple hypothesis testing problems. The traditional multiple testing research has focused on the offline setting, which means we have an entire batch of hypotheses and the corresponding p-values. In their seminal work,~\citet{benjamini1995controlling} developed a procedure (called {\em BH procedure}) to control FDR below a preassigned level for this offline setting.
However, the fact that offline FDR control techniques require aggregating p-values from all the tests and processing them jointly, makes it impossible to utilize them for a number of applications which are best modeled as an {\em online hypothesis testing} problem~\citep{foster2008alpha} (a more formal definition will be provided later).
In this scenario, we assume that an infinite sequence of hypotheses arrive sequentially in a stream, and decisions are made only based on previous results before next hypothesis arrives. In other words, the decisions have to be made without access to the number of hypotheses in the stream or the future p-values, but solely based on the previous decisions. For example, in marketing research a sequence of A/B tests can be carried out in an online fashion, or in a pharmaceutical drug test a sequence of clinical trial are conducted over time, or with publicly available datasets where new hypotheses are tested in an on-going fashion by different researchers.


\citet{foster2008alpha} designed the first online alpha-investing procedures that use and earn alpha-wealth in order to control a modified definition of FDR (referred to as {\em mFDR}), which was later extended to a class of {\em generalized alpha-investing} (GAI) rules by~\cite{aharoni2014generalized}. ~\citet{javanmard2015online,javanmard2018online} showed that monotone GAI rules, appropriately parametrized, can control the FDR
for independent p-values as opposed to the modified FDR controlled in~\citep{foster2008alpha,aharoni2014generalized}. Within this class, of special note here is an procedure (testing rule) called {\em LORD} that the authors noted performs consistently well in practice.~\citet{ramdas2017online} demonstrated a modification to the GAI class (called GAI++) that improved the earlier GAI algorithms (uniformly), and the improved LORD++ method arguably represents the current state-of-the-art in online multiple hypothesis testing. Very recently,~\citep{ramdas2018saffron} empirically demonstrated that using adaptiveness, some further improvements in the power over LORD++ can be obtained. We survey these online testing procedures in more detail in Section~\ref{sec:onlineFDR}.

All these above online testing procedures take p-values as input and decide at each time based on previous decision outcomes. However, these procedures ignore additional information that is often available
in modern applications. In addition to the p-value $P_i$, each hypothesis $H_i$ could also have a feature vector $X_i$ lying in some space $\cX \subseteq \mathbb{R}^d$. The feature vector encode auxiliary\footnote{Also sometimes referred to as {\em prior} or {\em side} information.} information related to the tested hypothesis, and is also often referred as {\em context}. The feature vector $X_i$ only carry indirect information about the likelihood of the hypothesis $H_i$ to be false but the relationship is not fully known ahead of time. For example, when conducting an A/B test for a logo size change in a website, contextual information such as text, layouts, images and colors in this specific page can be useful in making a more informative decision.  Similarly another example arises when testing whether a mutation is correlated with the trait, here contextual information about both the mutation and the trait such as its location, epigenetic status, etc.,  could provide valuable information that can increase the power of these tests.

While the problem of using auxiliary information in testing has been considered in offline setting~\citep{ignatiadis2016data, genovese2006false, li2016multiple, ramdas2017unified, xia2017neuralfdr, lei2018adapt}, in this paper we focus on the more natural online setting, where unlike the batch setting p-values and contextual features are not available at the onset, and a decision about a hypothesis should be made when it is presented. To the best of our knowledge, this generalization of the online testing problem has not been considered before. Our main contributions in this paper are as follows.
\begin{list}{{\bf (\arabic{enumi})}}{\usecounter{enumi}
\setlength{\leftmargin}{20pt}
\setlength{\listparindent}{\parindent}
\setlength{\parsep}{0pt}}
\item \textbf{Incorporating Contextual Information.} Building on GAI/GAI++ rules~\citep{aharoni2014generalized,ramdas2017online}, we propose a new broad class of powerful online testing rules, that we refer to as {\em contextual generalized alpha-investing} (CGAI) rules. This new class of rules incorporates the available contextual features in the testing process, and we prove that any monotone rule from this class can control online FDR under some standard assumption on p-values. More formally, we assume each hypothesis $H$ is characterized by a tuple $(P,X)$ where $P \in (0, 1)$ is the p-value, $X$ is the contextual feature vector. We use a very general model for these $X$'s, as a vector coming from some generic space $\cX \subseteq \mathbb{R}^d$. We consider a sequence of hypotheses $(H_1,H_2,\dots)$ that arrive sequentially in a stream one at each timestep $t=1,2,\dots$, with corresponding $((P_1,X_1),(P_2,X_2),\dots)$, our testing rule generates a sequence of significance levels $(\alpha_1,\alpha_2,\dots)$ at each timestep based on previous decisions and contextual information seen so far. Our decision rule for each hypothesis $H_t$ takes the form $\1 \{P_t \leq \alpha_t\}$, and under mutual independence of the p-values and independence between the p-values $P_t$'s and the contextual features $X_t$'s for null hypotheses, we show that any monotone rule from this class ensures that FDR remains below a preassigned level at any time. In the proof the sigma-field formalizes ``what we know'' at time $t$, and we define a filtration via the sigma-fields of both previous decisions and the contextual features. We also show that we can have mFDR control under a weaker assumption on p-values.

\item \textbf{Context Weighting.} While contextual generalized alpha-investing (CGAI) is a rich class of FDR control rules, we focus on a subclass of these rules for designing a practical online FDR control procedure and to compare the statistical power of various procedures. In particular, we focus on a subclass, that we refer to as {\em context-weighted generalized alpha-investing} (CwGAI) rules, where the contextual features are used for weighting the significance levels in an online setting. In particular, we take a parametric function $\omega(;\theta)$ for a set of parameters $\theta$, and at timestep $t$ use $\omega(X_t;\theta)$ as a real-valued weight on $\alpha_t$ generated through GAI rules, with the intuition that larger weights should reflect an increased willingness to reject the null. Since the parameter set $\theta$ is unknown, a natural idea here will be to learn it in an online fashion to maximize the number of empirical discoveries. 
This gives rise to a new class of online testing rules that incorporates the context weights through a learnt parametric function.

\item \textbf{Statistical Power Analysis.} Having established that our procedures control the false discovery rate, we next turn to the question of what effect context weighting has on the statistical power in an online setting. Lots of factors, such as frequency of true hypotheses, or on their order, affect the power. Hence, to make a rigorous comparison, we consider a standard mixture model used in such analyses where each null hypothesis is false with a fixed (unknown) probability~\citep{genovese2006false,javanmard2018online,lei2018adapt}, and focus on a slightly weaker variant of the popular online testing LORD procedure~\citep{javanmard2018online}. By considering a very general model of weighting where weights are random variables, and under the assumption that weights are positively associated with the null hypotheses being false, we derive a natural sufficient condition on the weights under which weighting improves the power in an online setting, while still guaranteeing FDR control. This is the first result that demonstrates the benefits of appropriate weighting in the online setting. Prior to this such results were only known in the offline setting~\citep{genovese2006false}.

\item \textbf{A Practical Procedure.} To design a practical online FDR control procedure with good performance, we model the parametric function $\omega(;\theta)$ as a neural network (multilayer perceptron), and train it in an online fashion to maximize the number of empirical discoveries. Our experimental evidence on a range of synthetic and real datasets show that our proposed procedure makes substantially more correct decisions compared to state-of-the-art  online testing procedures.
\end{list}



We formally describe the online multiple testing problem setup and review the published literature in this area in Section~\ref{sec:onlineFDR}. Here we start with a review of some relevant prior work in offline multiple testing.

\subsection{Related Work in the Offline Setting} 
In the offline setting, where we have access to the entire batch of p-values at one time instant, a number of procedures have been proposed to take advantage of the available auxiliary information to increase the power of test (to make more true discoveries). As we note below, the modeling of auxiliary information varies.

\citet{storey2002direct} proposed an adaptive FDR-control procedure based on estimating the proportion of true nulls from data. Reweighting the p-values by applying priors was considered by~\citep{benjamini1997multiple, genovese2006false, dobriban2016general, dobriban2015optimal}. In scenarios where priors are about spatial or temporal structure on hypotheses, \textit{Independent Hypothesis Weighting} procedure was proposed by \citet{ignatiadis2016data}, which clusters similar hypotheses into groups and assigns different weights to  these groups. \citet{hu2010false} utilized the idea of both grouping and estimating the true null proportions within each group. Some more procedures in \citep{barber2015controlling, g2016sequential, li2016accumulation, lei2016power} incorporate a prior ordering as the auxiliary information to focus on more promising hypotheses near the top of the ordering. This motivation underlies also the first online multiple testing paper of~\citep{foster2008alpha}.

{\em Structure-adaptive BH algorithm} (SABHA)~\citep{li2016multiple} and  {\em Adaptive p-value Thresholding} (AdaPT)~\citep{lei2018adapt} are two recent FDR control adaptive methods which derive the feature vector dependent decision rules. SABHA first censors the p-values below a fixed level, and then uses the censored p-values to estimate the non-null proportion (using non-parametric methods in practice), and then applies the weighted BH procedure of~\citep{genovese2006false}. AdaPT is based on adaptively estimating a Bayes-optimal p-value rejection threshold. At each iteration of AdaPT, an analyst proposes a significance threshold and observes partially censored p-values, then estimates the false discovery proportion (FDP) below the threshold, and proposes another threshold, until the estimated FDP is below the desired level. 

The offline testing algorithm mostly related to our results is the {\em NeuralFDR} procedure proposed by \cite{xia2017neuralfdr}, which uses a neural network to parametrize the decision rule. This procedure in the offline setting, with  access to all the p-values and the contextual features, comes up with a single decision rule $t(X)$ based on training a neural network for optimizing on the number of discoveries. In contrast, our method is in online multiple testing setup where we do not know all the p-values or the contextual features at once, and decision rules are different at each time, and varies as a function of previous outcomes and features. 

\subsection{Notation and Organization}
We denote $[n] = \{1,\dots, n\}$. Vectors are denoted by boldface letters. Given a sequence $(\Gamma_i)_{i \in \mathbb{N}}$, we denote by $\Gamma(n) = \sum_{i=1}^n \Gamma_i$ its partial sum.  
$\text{Bernoulli}(\cdot)$ and $\text{Uniform}(\cdot)$, with appropriate parameters, represent random variables drawn from the standard Bernoulli and uniform distributions respectively. $\cN(\mu,\sigma^2)$ denotes a random variable from normal distribution with mean $\mu$ and variance $\sigma^2$.

\paragraph{Organization.} The rest of the paper is organized as follows. In~\secref{onlineFDR}, we  introduce the online multiple testing problem and survey some of the prior work in this area.  In~\secref{CWFDR}, we propose a new  broad class of powerful testing procedures which can incorporate contextual (side) information of the hypothesis, and present our theoretical guarantee of online FDR control for this class.  In~\secref{weighted-rule}, we focus on a subclass of this above broad class, which are based on converting contextual features into weights, and using that to reweight the significance levels. In~\secref{power}, we theoretically show the increase in power those can be obtained by using weighted procedures in the online setting.  In~\secref{algorithm}, we design a practical algorithm for online multiple testing with contextual information,  and demonstrate its performance on synthetic and real datasets.


\section{Online False Discovery Rate Control: A Survey} \label{sec:onlineFDR}
We start with a review of the relevant notions in online multiple testing. The  online model was first introduced by~\citep{foster2008alpha}. Here we want to test an ordered (possibly infinite) sequence of hypotheses arriving in a stream, denoted by $\cH = (H_1, H_2, H_3, \dots, $), where at each timestep $t$ we have to decide whether to reject $H_t$ having only access to previous information. As is standard notation, $H_t \in \{0,1\}$ indicates if hypothesis $t$ is {\em null} ($H_t = 0$) or {\em alternative} ($H_t = 1$). Each hypothesis is associated with a p-value $P_t$. The results in this paper do not depend on the actual test used for generating the p-value. By definition of a {\em valid} p-value, if the hypothesis $H_t$ is {\em truly null}, then the corresponding p-value ($P_t$) is stochastically larger than the uniform distribution, i.e., 
\begin{align} \label{eqn:superp}
\mbox{\textbf{Super-uniformity of $P_t$:} If $H_t=0$ ($H_t$ is null) then } \Pr[P_t \leq u] \leq u, \,\, \mbox{ for all } u \in [0,1].
\end{align} 
No assumptions are made on the marginal distribution of the p-values for hypotheses that are non-nulls (alternatives). Although they can be arbitrary, they should be stochastically smaller than the uniform distribution, since only then do they carry signal that differentiates them from nulls. Let $\cH^0$ be the set $\{t \mid H_t = 0\}$ of indices of true null hypotheses in $\cH$ and let $\cH^1$ be the set $\{t \mid H_t = 1\}$ of remaining hypotheses in $\cH$.

An online multiple testing procedure is defined as a {\em decision rule} which provides a sequence of significance level $\{ \alpha_t \}$ and makes the following decisions:
\begin{align} \label{eqn:rts}
R_t = \1 \{P_t \le \alpha_t\} = \begin{cases} 
	1 & P_t \le \alpha_t \quad\;\;\;\, \Rightarrow \text{reject } H_t,\\ 
	0 & \text{otherwise} \quad \Rightarrow \text{accept } H_t.
\end{cases}
\end{align}
A rejection of the null hypothesis $H_t$ indicated by the event $R_t=1$ is also referred as discovery.

Let us start by defining the false discovery rate ($\fdr$), and true discovery rate ($\tdr$) formally in the online setting. For any time $T$, denote the first $T$ hypotheses in the stream by $\cH(T) = (H_1, \dots, H_T)$. Let $R(T) = \sum_{t= 1}^{T}R_t$ be the total number of discoveries (rejections) made by the online testing procedure till time $T$, and let $V(T) = \sum_{t \in \cH^0} R_t $ be the number of false discoveries (rejections) in them. Then online false discovery proportion (denoted as $\fdp$)  and the corresponding online false discovery rate (denoted as $\fdr$) till time $T$ are defined as:
\beq
\fdp(T) := \frac{V(T)}{R(T) \vee 1},  \quad  \quad  \fdr(T) := \E[\fdp(T)].
\eeq
Here, $R(T) \vee 1$ is the shorthand for $\max\{R(T),1\}$. The expectation is over the underlying randomness. The false discovery rate is the expected proportion of false discoveries among the rejected hypotheses. As mentioned earlier, this criterion was first introduced by~\citep{benjamini1995controlling} in the offline multiple testing scenario. Similarly, let $S (T) = \sum_{t \in \cH^1} R_t$ be the number of true discoveries (rejections) made by the online testing procedure till time $T$, and let $N_1(T)$ be the number of true alternatives (non-nulls).  Then online true discovery proportion (denoted as $\tdp$) and  online true discovery rate (denoted as $\tdr$) till time $T$ are defined as:
\beq
\tdp(T) := \frac{S(T)}{N_1(T) \vee 1}, \quad  \quad  \tdr(T) := \E[\tdp(T)].
\eeq
The true discovery rate is also referred as to as {\em power}.

In the online setting, the question that arises is how can we generate a sequence of $(\alpha_1,\alpha_2,\dots,\alpha_t,\dots)$ such that we have control over the false discovery rate ($\fdr$). Formally, the goal is to figure out a  sequence of significance levels $(\alpha_t)_{t \in \mathbb{N}}$ such that the $\fdr$ can be controlled under a given level $\alpha$ at any time $T \in \mathbb{N}$, i.e.,
\beq \label{goal}
\sup_{T} \; \fdr(T) \le \alpha.
\eeq
Note that none of these above four metrics (FDP, FDR, TDP, or TDR) can be computed without the underlying true labels (ground truth). 

A variant of FDR that arose from early works~\citep{foster2008alpha} on online multiple hypothesis testing is that of {\em marginal FDR} (denoted as mFDR), defined as:
$$\text{mFDR(T)}_{\eta} = \frac{\E[V(T)]}{\E[R(T)] + \eta}.$$
A special case of $\text{mFDR(T)} = \frac{\E[V(T)]}{\E[R(T)] + 1}$ is when $\eta = 1$. $\text{mFDR}$ measures the ratio of expected number of false discoveries to expected number of discoveries. Informally, while FDR controls a property of the realized set of tests, $\text{mFDR}$ is the ratio of the two expectations over many realizations. In general, the gap between FDR and mFDR can be very significant, and in particular controlling mFDR does not ensure controlling FDR at a similar level~\citep{javanmard2018online}. We will also provide a theoretical guarantee on mFDR control in a contextual setting under some weaker assumptions on p-values.

\paragraph{Generalized Alpha-Investing Rules.}\citet{foster2008alpha} proposed the first class of online multiple testing rules (referred to as alpha-investing rules) to control mFDR (under some technical assumptions).~\citet{aharoni2014generalized} further extended this class to generalized alpha-investing (GAI) rules, again to control mFDR. The general idea behind these rules is that you start with an initial wealth, which is under the desired control level. Some fraction of the wealth is spent (i.e., wealth decreases) for testing each hypothesis. However, each time a discovery occurs, a reward is earned, i.e., wealth increases for further tests. Building on these results,~\citet{javanmard2015online} demonstrated that monotone GAI rules, appropriately parameterized, can control the (unmodified) FDR for independent p-values. We now introduce the GAI rules conceptualized in these works.

Given a sequence of input p-values $(P_1, P_2,\dots )$, a generalized alpha-investing rule generates a sequence of significance levels $(\alpha_1,\alpha_2,\dots)$, which is then used for generating the sequence of decisions $R_t$'s as in~\eqref{eqn:rts}. In a GAI rule, the significance levels at time $t$ is a function of prior decisions (till time $t-1$):
\beq
\alpha_t = \alpha_t(R_1, \dots, R_{t-1}).
\eeq
In particular, this means that $\alpha_t$ does not directly depend on the observed p-values but only on past decisions. Let $\cF^{t} = \sigma (R_1, \dots, R_t)$ denote the sigma-field of decisions till time $t$. In GAI rules, we insist that $\alpha_t \in \cF^{t-1}$. 

Any GAI rule begins with an wealth of $W(0) > 0$, and keeps track of the wealth $W(t)$ available after $t$ testing steps. Formally, a generalized alpha-investing rule is specified by three (sequences of) positive functions $\alpha_t,\phi_t,\psi_t \in \cF^{t-1}$. Here, $\phi_t$ is the penalty of testing a new hypothesis $H_t$, and $\psi_t$ is the reward for making a rejection (discovery) at time $t$. In other words, a part of the wealth is used to test the $t$th hypothesis at level  $\alpha_t$, and the wealth is immediately decreased by an amount $\phi_t$. If the $t$th hypothesis is rejected, that is if $R_t= 1$, then an extra wealth equaling an amount $\psi_t$ is added to the current wealth. This can be explicitly stated as:
\begin{align}
	& W(0) = w_0, \\
	& W(t) = W(t - 1) - \phi_t + R_t \cdot \psi_t,
\end{align}
where $0 < w_0 < \alpha$ is the initial wealth. The parameters $w_0$ and the nonnegative sequences $\alpha_t,\phi_t,\psi_t$ are all user-defined. A requirement is that the total wealth $W(t)$ is always non-negative, and hence $\phi_t \le W(t-1)$. If the wealth ever equals zero, the procedure is not allowed to reject any more hypotheses since it has to set $\alpha_t=0$ from then on. An additional restriction is needed from the goal to control FDR, in that whenever a rejection takes place, $\psi_t$ should be bounded. Formally, this constraint is defined as:
\begin{equation} \label{eqn:psi}
	\psi_t \le \min \{\phi_t + b_t, \frac{\phi_t}{\alpha_t} + b_t -1\}.
\end{equation}
\citet{javanmard2015online} defined $b_t$ as a user-defined constant $B_0 > 0$, setting it to $\alpha - w_0$. Recently,~\citet{ramdas2017online} demonstrated that setting $$b_t = \alpha - w_0 \1 \{\rho_1 > t-1\}$$ could potentially lead to larger statistical power. Here, for a positive integer $k$, 
$$\rho_k := \min_{i \in \mathbb{N}} \, \{\sum_{t=1}^i R_t = k\},$$ 
is the time of $k$th rejection (discovery).~\citet{ramdas2017online} refer to this particular setting of $b_t$ as GAI++ rules. Unless otherwise specified, we use this improved setting of $b_t$ throughout this paper.  

Another important property when dealing with GAI rules is that of monotonicity. As studied by~\citep{javanmard2015online, javanmard2018online} in the context of GAI rules, and as is predominantly the case in offline multiple testing, monotonicity is the following condition:
\begin{align} \label{eqn:mono}
\mbox{if $\tilde{R_i} \le R_i$ for all $i \le t-1$ then } \;\; \alpha_t(\tilde{R}_1, \dots, \tilde{R}_{t-1}) \le \alpha_t (R_1, \dots, R_{t-1}).
\end{align}

While these online FDR procedures are widely used, a major shortcoming of them is that they ignore additional information that is often available during testing. Each hypothesis, in addition to the p-value, could have a feature vector which encodes contextual information related to the tested hypothesis. For example, in genetic association studies, each hypothesis tests the correlation between a variant and the trait. We have a rich set of features for each variant (e.g., its location, conservation, epigenetics, etc.) which could inform how likely the variant is to have a true association.

In this paper, we introduce the problem contextual online multiple testing, which captures the presence of this auxiliary information in modern applications.

\paragraph{\LORD (LORD) Rules.} One specific subclass of GAI rules (proposed by~\citep{javanmard2015online,javanmard2018online}) that arguably is state-of-the-art in online multiple hypothesis testing and performs consistently well in practice is known as \textit{\LORD} (LORD). We will focus on weighted variants of LORD later in this paper when we discuss about statistical power of tests. In fact, we will consider the recently improved LORD++ rules (proposed by~\citep{ramdas2017online}) that achieves the same or better statistical power than the LORD rules (uniformly).

The idea behind LORD (and LORD++) rules is that the significance level $\alpha_t $ is a function based only on {\em most recent discovery time}. Formally, we start with any sequence of nonnegative numbers $\gamma = (\gamma_t)_{t = 1} ^ {\infty}$, which is monotonically non-increasing  with $\sum_{t = 1}^{\infty} \gamma_t = 1$. At each time $t$, let $\tau_t$ be the last time a discovery was made before $t$, i.e., 
$$\tau_t := \max \{i \in \{1,\dots, t-1\}: R_i = 1 \},$$
with $\tau_t =0$ for all $t$ before the first discovery. The LORD rule defines $\alpha_t,\phi_t,\psi_t$ in the following generalized alpha-investing fashion.\footnote{Note that~\citet{javanmard2018online} defined three versions of LORD that slightly vary in how they set the significance levels. In this paper, we stick to one version, though much of the discussion in this paper also holds for the other versions.}

\noindent\fbox{%
\parbox{\textwidth}{%
\LORD (LORD)~\citep{javanmard2018online,javanmard2015online}:
\begin{eqnarray} 
	& W(0) = w_0,& \nonumber \\ 
	& \phi_t = \alpha_t = \begin{cases} 
	\gamma_{t} w_0 & \mbox{if $t \leq \rho_1$} \\ 
	\gamma_{t - \tau_t} b_0 & \mbox{if $t > \rho_1$},
\end{cases}&\nonumber \\
	& \psi_t = b_0,& \nonumber\\
	& b_0 = \alpha - w_0.& \nonumber
\end{eqnarray}
}}
Typically, we will set $w_0 = \alpha / 2$, in which case, the above rule could be simplified as $\phi_t = \alpha_t = \gamma_{t - \tau_t} b_0 = \gamma_{t - \tau_t} \alpha/2$. 

As with any GAI rule,~\citep{ramdas2017online} showed that one could replace $b_0$ with $b_t = \alpha - w_0 \1 \{\rho_1 > t-1\}$ to achieve potentially better power, while still achieving online FDR control at level $\alpha$. With this replacement, we defined LORD++ as follows.

\noindent\fbox{%
\parbox{\textwidth}{%
Improved \LORD (LORD++)~\citep{ramdas2017online}:
\begin{eqnarray} 
	& W(0) = w_0 \ge \alpha/2,& \nonumber\\ 
	& \phi_t = \alpha_t = \gamma_{t - \tau_t} b_t,& \nonumber \\
	& \psi_t = b_t =  \alpha - w_0 \1 \{\rho_1 > t-1\}.& \nonumber
\end{eqnarray}
}}

It can be easily observed that both LORD and LORD++ rules satisfy the monotonicity condition from~\eqref{eqn:mono}.

\paragraph{SAFFRON Procedure.}This is a very recently proposed online FDR control procedure by~\citet{ramdas2018saffron}. The main difference between SAFFRON (Serial estimate of the Alpha Fraction that is Futilely Rationed On true Null hypotheses) and the previously discussed LORD/LORD++ procedures comes in that SAFFRON is an adaptive method, based on adaptively estimating the proportion of true nulls. SAFFRON can be viewed as an online extension of Storey's adaptive version of BH procedure in the offline setting. SAFFRON does not belong to the GAI class, whose extension to the contextual online setting is the main focus of this paper. See Appendix~\ref{app:saffron} for more details about SAFFRON, and the experiments with SAFFRON that suggests that contextual information could potentially help here too.


\section{Contextual Online FDR Control} \label{sec:CWFDR}
We start with an informal definition of the contextual online multiple testing problem. Consider a setting where we test an ordered (possibly infinite) sequence of null hypotheses, denoted $\cH = (H_1, H_2, H_3, \dots, $).  Each hypothesis $H_t$ is associated with a p-value $P_t \in (0,1)$  and a vector of contextual features $X_t \in \cX \subseteq \mathbb{R}^d$, thus can be represented by a tuple $(H_t, P_t, X_t)$. We observe $P_t$ and $X_t$, but do not know $H_t$. The goal of contextual online testing, is at each step $t$, decide whether to reject $H_t$ having only access to previous decisions and contextual information seen so far. The overall goal is to control online FDR under a given level $\alpha$ at any time and improve the number of useful discoveries by using the contextual information. All missing details in this section are collected in Appendix~\ref{app:CWFDR}.

Under the alternative, we denote the density distribution (PDF) of p-values as $f_1(p\mid X)$ (for $X \in \cX$), and the cumulative distribution (CDF) of p-values as $F_1(p\mid X)$. Here $f_1(p\mid X)$ can be any {\em arbitrary unknown} function, as long as the p-values are stochastically smaller than those under the null. Note that $f_1(p\mid X)$ is not identifiable from the data as we never observe $H_t$'s directly. This can be illustrated through a simple example described in Appendix~\ref{app:CWFDR}. Let us now formally define the contextual online FDR control problem.
\begin{Def} [Contextual Online FDR Control Problem]
Given a (possibly infinite) sequence of $(P_t,X_t)$'s ($t \in \mathbb{N}$) where $P_t \in (0,1)$ and $X_t \in \cX$, the goal is to generate a significance levels $\alpha_t's$ as a function of prior decisions and contextual features
\beq
\alpha_t = \alpha_t(R_1, \dots, R_{t-1}, X_1, \dots, X_{t}),
\eeq
and a corresponding set of decisions
\beq 
R_t = \begin{cases} 
1 & P_t \le \alpha_t (R_1, \dots, R_{t-1}, X_1, \dots, X_t), \\ 
0 & \text{otherwise.}
\end{cases}
\eeq
such that $\sup_{T} \; \fdr(T) \le \alpha$.
\end{Def}
A related definition would be to control mFDR (instead of FDR) under level $\alpha$. Note that in the contextual setting, we consider the significance levels to be functions of prior decisions and contextual features seen so far. This differs from existing online multiple testing rules described in the previous section where the significance levels ($\alpha_t$'s) depend only on prior results ($\alpha_t \in \cF^{t-1}$), i.e., does not use any contextual information. 

In the presence of contextual information, we use the sigma-field of decisions till time $t$ as $\cF^{t} = \sigma (R_1, \dots, R_t)$, and the sigma-field of features till time $t$ as $\cG^t = \sigma(X_1, \dots, X_t)$. Our first contribution in this paper is to define a contextual extension of GAI rules, that we refer to as \textit{\CGAI} (contextual GAI or CGAI) rules. A contextual GAI rule is defined through three functions, $\alpha_t, \phi_t, \psi_t \in \sigma(\cF^{t-1} \cup \cG^t)$, that are all computable at time $t$. 

A valid contextual GAI rule is required to satisfy the following conditions:

\noindent\fbox{%
\parbox{\textwidth}{%
Contextual GAI:
\begin{eqnarray}
	& \textbf{Initial wealth: } W(0) = w_0, \text{with $0 < w_0 < \alpha$}, \label{eq1}& \\
	& \textbf{Wealth Update: } W(t) = W(t - 1) - \phi_t + R_t \cdot \psi_t, \label{eq2}&\\
	& \textbf{Non-negativity: } \phi_t \le W(t-1),\label{eq3}& \\
	& \textbf{Upper bound on reward: } \psi_t \le \min \{\phi_t + b_t, \frac{\phi_t}{\alpha_t} + b_t -1\}. \label{eq4}&
\end{eqnarray}
We set $b_t = \alpha - w_0 \1 \{\rho_1 > t-1\}$ as proposed by~\citep{ramdas2017online}. 
}}
Also as with GAI rules, we can add a monotonicity property to a contextual GAI rule as follows:
\begin{multline} \label{eq5}
\textbf{Monotoncity: } \text{If $\tilde{R_i} \le R_i$ for all $i \le t-1$ then }\\ \alpha_t(\tilde{R}_1, \dots, \tilde{R}_{t-1}, X_1, \dots, X_t) \le \alpha_t (R_1, \dots, R_{t-1}, X_1, \dots, X_t), \mbox{ for any fixed $\bX^t = (X_1, \dots, X_t)$}.
\end{multline}
A contextual GAI rule satisfying the monotonicity condition is referred to as {\em monotone contextual GAI}.

Our first result establishes the online FDR control for any monotone contextual GAI rule under independence of p-values, and the mutual independence of p-values and contextual features under the null.\!\footnote{Note that a standard assumption in hypothesis testing is that the p-values under the null are uniformly distributed in $(0,1)$, which implies the mutual independence of p-values and contextual features under the null.} 

We start by presenting the following lemma, which is an intermediate result for the proof of FDR later. Note that $R_t = \1 \{P_t \le \alpha_t\}$, where $\alpha_t = \alpha_t(R_1, \dots, R_{t-1}, X_1, \dots, X_t)$ is $\sigma(\cF^{t-1} \cup \cG^t)$-measurable and is a coordinatewise non-decreasing function of $R_1,\dots, R_{t-1}$ for any fixed $\bX^t = (X_1, \dots, X_t)$. Due to the super-uniformity of p-values under the null~\eqref{eqn:superp}, we immediately have that for independent p-values under the null the following conditional super-uniformity condition holds. 
\begin{align} \label{eqn:super-uniformity}
 \Pr[P_t \le \alpha_t  \mid \sigma(\cF^{t-1} \cup \cG^t) ] \le \alpha_t.
\end{align}
\lemref{super-uniform} states a more general result about super-uniformity of independent p-values under the null.  These independence assumptions are standard in multiple testing literature (see, e.g.,~\cite{ramdas2017online,javanmard2018online,xia2017neuralfdr} among others). The proof is based on a {\em leave-one-out} technique, a variant of which was also used by~\citet{ramdas2017online} (and also by~\citet{javanmard2018online}) in their analyses. The main distinction for us comes in that we consider the sigma-field at each time $t$ as $\sigma(\cF^{t-1} \cup \cG^t)$ including the information of contextual features till time $t$, instead of just $\cF^{t-1}$.
\begin{lem}[Super-uniformity]\label{lem:super-uniform}
	Let $g: \{0, 1\}^{T} \to \mathbb{R}$ be any coordinatewise non-decreasing function such that $g(\bR) > 0$ for any vector $\bR \neq (0, \dots, 0)$. Then for any index $t \le T$ such that $t \in \cH^0$, we have
	\begin{align*}
		&\E \bigg [\frac{\1 \{P_t \le \alpha_t(R_1, \dots, R_{t-1}, X_1, \dots, X_t)\}}{g(R_1, \dots, R_T) \vee 1}  \bigg \vert \sigma(\cF^{t-1} \cup \cG^t) \bigg ] \le \E \bigg [ \frac{\alpha_t (R_1, \dots, R_{t-1}, X_1, \dots, X_T)}{g(R_1, \dots, R_T) \vee 1} \bigg \vert \sigma(\cF^{t-1} \cup \cG^t) \bigg ].
	\end{align*}
\end{lem}
The following theorem establishes the FDR property on any monotone contextual GAI rule under some independence assumptions on the p-values. Note that, as mentioned above, the p-values ($P_t$) could be arbitrary related to the contextual features ($X_t$) under the alternative (when $H_t=1$). The proof considers two cases: (a) $H_t=0$ in which case it uses the super-uniformity statement established in~\lemref{super-uniform} and the fact that $\psi_t \le \phi_t/\alpha_t + b_t -1$ (by definition), and (b) $H_t=1$ in which case we use the fact that $\psi_t \le \phi_t + b_t$ (by definition). 
\begin{thm} \label{thm:fdr-control} [FDR Control]
Consider a sequence of $((P_t,X_t))_{t \in \mathbb{N}}$  of p-values and contextual features. If the p-values $P_t$'s are independent, and additionally $P_t$ are independent of all $(X_t)_{t \in \mathbb{N}}$ under the null $($whenever $H_t=0)$, then for any monotone contextual generalized alpha-investing rule $($i.e., satisfying conditions~\eqref{eq1},~\eqref{eq2},~\eqref{eq3},~\eqref{eq4}, and~\eqref{eq5}$)$, we have online FDR control,
	\beq
	\sup_{T \in \mathbb{N}} \; \text{\em FDR}(T) \le \alpha.
	\eeq
\end{thm}

Turning our attention to mFDR, we can prove a guarantee for mFDR control under a weaker condition than that in \thmref{fdr-control}. In particular, one can relax the independence assumptions to a weaker conditional super-uniformity assumption from~\eqref{eqn:super-uniformity} (as mentioned above, the validity of p-values~\eqref{eqn:superp} and independence implies conditional super-uniformity). We define that a null p-value is conditionally super-uniform on past discoveries and contextual features so far, if the following holds
\begin{align} \label{eqn:condsup}
\textbf{Conditional super-uniformity of $P_t$: } \text{If $H_t=0$ ($H_t$ is null) then } \Pr[P_t \le \alpha_t  \mid \sigma(\cF^{t-1} \cup \cG^t) ] \le \alpha_t.
\end{align}
Our next theorem formalizes the mFDR control for any contextual GAI rule (not necessarily monotone) under the above condition.
\begin{thm} \label{thm:mfdr-control} [mFDR Control]
Consider a sequence of $((P_t,X_t))_{t \in \mathbb{N}}$  of p-values and contextual features. If the p-values $P_t$'s are conditionally super-uniform distributed (as in~\eqref{eqn:condsup}), then for any contextual generalized alpha-investing rule $($i.e., satisfying conditions~\eqref{eq1},~\eqref{eq2},~\eqref{eq3}, and~\eqref{eq4}$)$, we have online mFDR control,
\beq
\sup_{T \in \mathbb{N}} \; \text{\em mFDR}(T) \le \alpha.
\eeq
\end{thm}

\begin{rem}
For arbitrary dependent p-values and contextual features, the FDR control can also be obtained by using a modified LORD rule defined in~\cite{javanmard2018online}, under a special case where the contextual features are transformed into weights satisfying certain conditions. See Proposition~\ref{prp:FDR-dependent-general-weight}  for a formal statement.
\end{rem}

\section{Context-weighted Generalized Alpha-Investing Rules} \label{sec:weighted-rule}
The contextual GAI rules, introduced in the previous section, form a very general class of online multiple testing rules. In this section, we focus on a subclass of these rules, which we refer to as \textit{\CwGAI} (context-weighted GAI or CwGAI) rules. Specifically, we consider $\alpha_t$ to be a product of two functions with the first one of previous decisions and second one based on the current contextual feature. Borrowing from offline multiple testing literature, we refer to the second parametric function as a {\em weight function} (denoted by $\omega(\cdot)$). Therefore, a context-weighted GAI rule has this form, 
\begin{align} \label{Context-weightedGAI}
\alpha_t(R_1,\dots,R_{t-1},X_1,\dots,X_t) := \alpha_t(R_1, \dots, R_{t-1}) \cdot \omega (X_t; \theta),
\end{align}
where $\omega(X_t;\theta)$ is a parametric function with a parameter set $\theta \in \Theta$. Since CwGAI is a subclass of CGAI rules, the above FDR and mFDR control theorems are valid for this class too. 

Our reasons for considering this subclass of contextual GAI rules include: \begin{enumerate*} \item With context-weighted generalized alpha-investing rules we obtain a simpler form of $\alpha_t$ by separating the contextual features from that of previous outcomes, making it easier to design functions that satisfy the monotonicity requirement of the GAI rules,
\item It is now convenient to model the weight function by any parametric function, and
\item We can learn the parameters of the weight function empirically by maximizing the number of discoveries. This forms the basis of a practical algorithm for contextual online FDR control that we describe in Section~\ref{sec:algorithm}.\end{enumerate*}
Note that the GAI rules are still a subclass of context-weighted generalized alpha-investing rules by taking the weight function as a constant $1$. We illustrate the relationship among various classes of testing rules in \figref{relation}. 

The idea of weighting p-values by using prior information has been widely studied in offline multiple testing setup~\citep{genovese2006false, ignatiadis2016data, li2016multiple, lei2018adapt, xia2017neuralfdr, ramdas2017unified}. In many applications, contextual information can provide some prior knowledge about the true underlying state at current time, which may be incorporated in by a weight $\omega_t = \omega(X_t;\theta)$. Intuitively, the weights indicate the strength of a prior belief whether the underlying hypothesis is null or not. Therefore, a larger weight $\omega_t > 1$ provides more belief of a hypothesis being an alternative which makes the procedure to reject it more aggressively, while a smaller weight $\omega_t < 1$ indicates a higher likelihood of a null which makes the procedure reject it more conservatively. 
\begin{figure}\centering
	\includegraphics[width = 15cm]{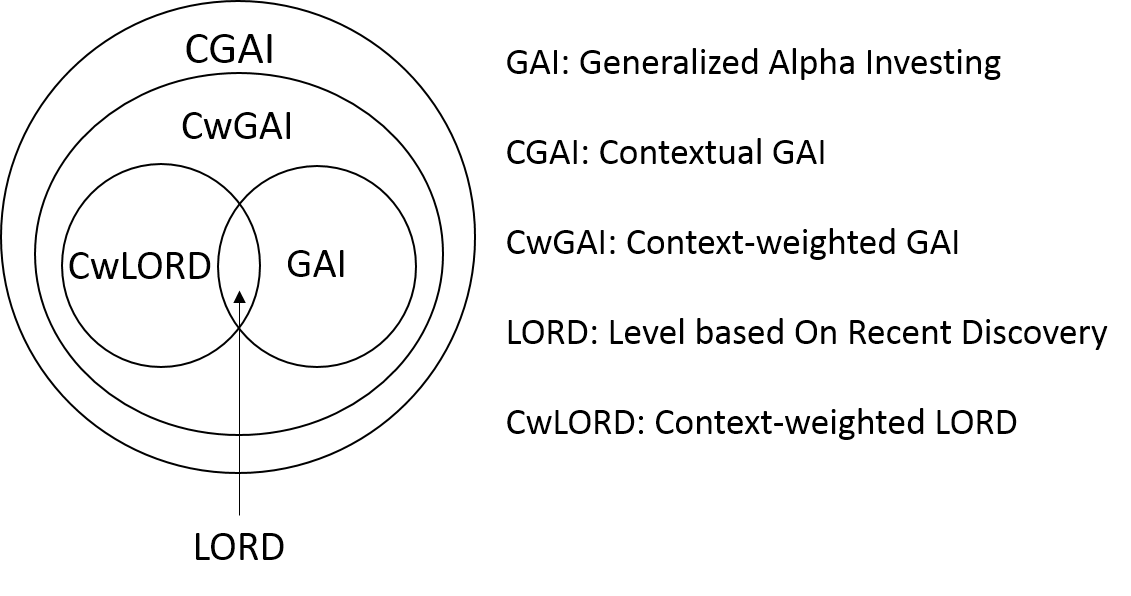}
	\caption{Relationship among various testing rules. One could replace LORD with LORD++ (resp.\ CwLORD with CwLORD++).}
	\label{fig:relation}
\end{figure}

\paragraph{Weighting in Online vs.\ Offline Setting with FDR Control.} In the offline setting with a fixed set $\{P_1,\dots,P_n\}$ of p-values, the idea of using prior weights is as follows: a) first the given set of weights (say, $\hat{\omega}_i'$s) are rescaled to make the average equal to one giving rise to new $\omega_i$'s ($i \in [n]$), and b) then any existing offline FDR control algorithm is applied by replacing $P_i$ with the weighted p-value $P_i/\omega_i$. This p-value weighting in the offline setting was first suggested by~\citep{genovese2006false}, and we refer the reader to that paper for a detailed account on p-value weighting. However, in an online setting, the weights $\omega_t$'s are computed at each timestep $t$ without access to the total number of hypothesis or contextual information ahead of time, which means weights cannot be rescaled to have unit mean in advance. 

Instead of weighting the p-values, as represented in~\eqref{Context-weightedGAI}, we consider weighting the significance levels $\alpha_t$'s, as was also considered by~\citet{ramdas2017unified}. Note that weighting p-values is equivalent to weighting significance levels in terms of decision rules conditioning on the same significance levels, i.e., given the same $\alpha_t$'s, we have $\{P_t/\omega_t \leq \alpha_t \} \equiv \{P_t \le  \alpha_t \omega_t \}$ for all $t$. The main difference is that when $\alpha_t$'s are weighted, the penalty $\phi_t$'s and rewards $\psi_t$'s also needs to be adjusted according to the GAI constraints. For example, as dictated by~\eqref{eqn:psi}, if we overstate our prior belief in the hypothesis being alternative by assigning a large $\omega_t > 1$, the penalty will need to be more or the reward will need to be less.  

\paragraph{Context-weighted LORD++.} Next let us see how the principle of context weighting can be applied to LORD++  rules defined in Section~\ref{sec:onlineFDR}. Given a weight function, $\omega : \cX \times \Theta \rightarrow \mathbb{R}$, we define the context-weighted LORD++ (CwLORD++) testing rule as follows.

\noindent\fbox{%
\parbox{\textwidth}{%
Context-weighted LORD++ (CwLORD++):
\begin{eqnarray*}
	& W(0) = w_0,& \\
	& \phi_t = \alpha_t = \min \{\gamma_{t - \tau_t} b_t \cdot \omega(X_t; \theta), W(t-1) \},&\\
	& \psi_t = b_t = \alpha - w_0 \1 \{\rho_1 > t-1\}.&
\end{eqnarray*}
}}
Similarly, given the function $\omega$, we can also define context-weighted LORD using the LORD rules from Section~\ref{sec:onlineFDR}. As pointed out before, when $\alpha_t$'s are reweighted, the penalty $\phi_t$'s are also adjusted accordingly. This provides a clean way of incorporating the weights in an online setup without having to rescale the weights to have unit mean.

In the next section, we see how weighting could be beneficial in an online setting. In Section~\ref{sec:algorithm}, we discuss an online multiple testing procedure based on CwLORD++ (see Algorithm~\ref{alg:NeuralOnlineFDR}), where we use a neural network to model the weight function $\omega$ with parameters trained to adaptively increase the discovery proportion over time. Our experimental investigation will be based on this procedure.

\section{Statistical Power of Weighted Online Rules} \label{sec:power}
In this section, we answer the question whether weighting helps in an online setting in terms of increased power. We answer this question in affirmative, in the context of the popular LORD procedure of~\citet{javanmard2018online}. The benefits of weighting in the offline setting, in terms of increased power was first studied by \cite{genovese2006false}, who showed that a weighted BH procedure improves the power over the corresponding unweighted procedure if weighting is {\em informative}, which roughly means that the weights are positively associated with the non-nulls.  Missing details from section are collected in Appendix~\ref{app:power}.

We use a model setup as in \citep{genovese2006false} modified to an online setting. In particular, we consider a mixture model where each null hypothesis is false with probability $\pi_1$ independently of other hypotheses, and the p-values corresponding to different hypotheses are all independent. While the mixture model is admittedly {\em idealized}, it does offer a natural ground for comparing the power of various testing procedures~\citep{genovese2006false,javanmard2018online}. Let us start with a formal description of the mixture model. The rest of the discussion in this section will be with respect to this mixture model.

\paragraph{Mixture Model.} For any $t \in \mathbb{N}$, let
\begin{eqnarray*}
& H_1, \dots, H_t \stackrel{\text{i.i.d.}}{\sim} \text{Bernoulli} (\pi_1), & \\
& X_t \mid H_t = 0 \sim \cL_0(\cX), \quad  X_t \mid H_t = 1 \sim \cL_1(\cX),  &\\
& P_t \mid H_t = 0, X_t \,  \sim \text{Uniform} (0,1), & \\
& P_t \mid H_t = 1, X_t  \, \sim F_1 (p \mid X_t). &
\end{eqnarray*}
where $0 < \pi_1 < 1$ and where $\cL_0(\cX)$, $\cL_1(\cX)$ are two probability distribution on the contextual feature space $\cX$. Let $F = \int F_1(p \mid X) \mathrm{d} \cL_1(\cX)$ be the marginal distribution of p-value under alternative. Marginally, the p-values are i.i.d.\ from the CDF $G(a) = (1-\pi_1)U(a) + \pi_1 F(a)$, where $U(a)$ is the CDF of Uniform(0,1). For this mixture model, a lower bound (and in some cases tight) bounds on the power of the LORD procedure was established by~\citep{javanmard2018online}. Our aim will be to compare this with the power of a weighted version of LORD to establish a separation.

\paragraph{General Weighting Scheme.}  We consider the model of general weighting where the weight $\omega_t$ is a random variable that is conditionally independent of $P_t$ given $H_t$ for all $t = 1 \dots, \infty$. More concretely, we assume that the weights have different marginal distributions under null and alternative,
\begin{align} \label{weight}
\omega\mid H_t = 0 \sim Q_0,  \quad \omega\mid H_t = 1 \sim Q_1,
\end{align}
for some unknown probability distributions $Q_0, Q_1$ that are continuous on $(0, \infty)$. For the power analysis, we draw the weight 
\begin{align} \label{eqn:w}
\omega_t \stackrel{\text{i.i.d.}}{\sim} (1-\pi_1)Q_0 + \pi_1 Q_1,
\end{align} 
with $P_t$ and $\omega_t$ being conditionally independent given $H_t$ for all $t =1 \dots, \infty$.  

\paragraph{Contextual Weighting Scheme.} This framework of weighting in~\eqref{weight} is very general. For example, it includes as a special case, the following contextual weighting scheme, where we assume that there exists a weight function of contextual features $\omega\,:\, \cX \rightarrow \mathbb{R}$, and the distributions of weights are defined as:
\begin{align} \label{eqn:sep}
\omega \mid H_t = 0 \, \sim \omega(X), \text{ with } X \sim \cL_0, \quad \omega \mid H_t = 1 \, \sim \omega (X), \text{ with } X \sim \cL_1.
\end{align}
Now $Q_0$ and $Q_1$ in~\eqref{weight} are defined as the distributions of $\omega (X)$ under the null and alternative, respectively.  Given $Q_0$ and $Q_1$, the weight $\omega_t$ is sampled as in~\eqref{eqn:w}.\footnote{In case , $X_t \stackrel{\text{i.i.d.}}{\sim} (1-\pi_1)\cL_0 + \pi_1 \cL_1$, then one can define  $\omega_t$ directly as $\omega_t= \omega(X_t)$ with $Q_0$ and $Q_1$ defined as the distributions of $\omega(X_t)$ under the null and alternative, respectively.}
We do not require that the contextual features $X_t$'s be independent, but only that they be identically distributed as $\cL_0$ (under null) or $\cL_1$ (under alternative). A reader might notice that that while the distributions $Q_0$ and $Q_1$ for weights are defined through $X_t$'s distribution, the weight $\omega_t$ is sampled i.i.d.\ from the mixture model $(1-\pi_1)Q_0 + \pi_1 Q_1$, regardless of the value of $X_t$. 

Note that the independence assumption on p-values can still be satisfied even when the $X_t's$ are dependent.\!\footnote{For example, in practice it is common that the contextual features are dependent (e.g., same genes or genetic variants may be tested in multiple independent experiments at different time), but as long as the tests are carried out independently the p-values are still independent.} Since this contextual weighting scheme is just a special case of the above general weighting scheme, in the remainder of this section, we work with the general weighting scheme.

\paragraph{Informativeness.}  Under~\eqref{weight}, the marginal distribution of $\omega$ is $Q = (1-\pi_1) Q_0 + \pi_1 Q_1$. For $j = 0,1$, let $ u_j = \E[\omega \mid H_t = j]$ be the means of $Q_0$ and $Q_1$ respectively. We also assume that the following notion of informative-weighting, based on a similar condition used by~\citep{genovese2006false} in the offline setting.
\begin{align} \label{weight-condition}
\textbf{Informative-weighting: } u_0 < 1, \quad u_1 > 1, \quad u = \E[\omega] = (1-\pi_1) u_0 + \pi_1 u_1 = 1.
\end{align}
\begin{remark}
Informative-weighting places a very natural condition on the weights. Roughly it means that the weight should be positively associated to the true alternatives, or equivalently, the weight under alternative is more likely to be larger than that under the null. The marginal mean of weight $\E[\omega]$ is not necessary to be one. But for the theoretical comparison of the power of different procedures, it is convenient to scale the weight to have unit mean so that we can use the p-value reweighting akin to the offline setting. For empirical experiments, we will use an instantiation of CwLORD++ (see Section~\ref{sec:experiments}), that does not require the weight to have mean one. 
\end{remark}
Let the weights $\omega_t$'s be random variables that are drawn i.i.d.\ from this mixture model with marginal distribution $Q$. Taking the LORD++ rule from Section~\ref{sec:onlineFDR}, we define a weighted LORD++ rule as follows.
\begin{Def} [Weighted LORD++] \label{def:weightedlord++}
Given a sequence of p-values, $(P_1,P_2,\dots)$ and weights $(\omega_1,\omega_2,\dots)$, apply LORD++ with level $\alpha$ to the weighted p-values $(P_1/\omega_1, P_2/\omega_2, \dots)$.
\end{Def}

We want to emphasize that the weighted LORD++ rule actually reweights the p-values (as done in the offline weighted BH procedure~\citep{genovese2006false}) and then applies the original LORD++ to these reweighted p-values. So this is slightly different from the idea of CwLORD++, that we mentioned above, which reweights the significance levels and then applies it to the original p-values. To understand the difference, let us start from their definitions.

In LORD++, the penalty is $\phi_t = \alpha_t = \gamma_{t - \tau_t} b_t < W(t-1)$, which is always less than current wealth due to the construction of $\gamma_t$'s and $w_0$. So in weighted LORD++, we are comparing reweighted p-value $P_t'$ to the level $\alpha_t = \gamma_{t - \tau_t} b_t$, which is equivalent to comparing the original p-value $P_t$ to the level $\alpha_t' = \gamma_{t - \tau_t} b_t \omega_t$.

On the other hand, in CwLORD++, we take a penalty of the form $\tilde{\phi}_t = \tilde{\alpha}_t = \min \{\gamma_{t - \tau_t} b_t \omega_t, W(t-1) \}$. We take the minimum of reweighted significance level and the current wealth, to prevent the penalty $\gamma_{t - \tau_t} b_t \omega_t$ from exceeding the current wealth which would violate a tenet of the alpha-investing rules.

A simple corollary is that the actual significance levels used in weighted LORD++ are greater than those in CwLORD++, i.e., 
$$\alpha_t' = \gamma_{t - \tau_t} b_t \omega_t \ge \min \{\gamma_{t - \tau_t} b_t  \omega_t, W(t-1)\} = \tilde{\alpha}_t.$$ 
That implies the power of weighted LORD++ is equal to or greater than the power of CwLORD++, whereas the FDR of weighted LORD++ may also be higher than that of CwLORD++. From Theorem~\ref{thm:fdr-control}, we know that we have FDR control with CwLORD++, however that result does not hold for weighted LORD++ (as weighted LORD++ is not strictly a contextual GAI rule) 
We now show that the above weighted LORD++ can still control online FDR at any given level $\alpha$ under the condition $\E[\omega] = 1$, which we do in the following proposition.

\begin{prp} \label{prp:FDR-general-weight}
Suppose that the weight distribution satisfies the informative-weighting assumption in~\eqref{weight-condition}. Suppose that p-values $P_t$'s are independent, and are conditionally independent of the weights $\omega_t$'s given $H_t$'s. Then the weighted LORD++ rule can control the online FDR at any given level $\alpha$, i.e.,
\beq
\sup_{T \in \mathbb{N}}\; \text{\em FDR}(T) \le \alpha.
\eeq
\end{prp}

\paragraph{Weakening the Assumptions from Proposition~\ref{prp:FDR-general-weight}.} In most applications, the independence between p-values and weights needed in Proposition~\ref{prp:FDR-general-weight} is not guaranteed.~\citet{javanmard2018online} achieved the FDR control under dependent p-values by using a modified LORD rule, which sets $\psi_t = b_0$ and $\alpha_t = \phi_t = \gamma_t W(\tau_t)$ with the fixed sequence $(\gamma_t)$ satisfying $\sum_{t = 1}^{\infty} \gamma_t (1 + \log(t)) \le \alpha / b_0$.

We can extend the FDR control results to the dependent weighed p-values. In particular, as long as the following condition is satisfied, i.e., for each weighted p-value $(P_t /\omega_t)$ marginally
\begin{align} \label{eqn:depend}
\Pr[P_t/\omega_t \le u \mid H_t = 0] \le u, \quad \text{for all } u \in [0,1], 
\end{align}
then the upper bound of FDR stated in Theorem 3.7 in \cite{javanmard2018online} is valid for weighted p-values. Specifically, if the modified LORD rule in Example 3.8 of \cite{javanmard2018online} is applied to the weighted p-values under the assumption in~\eqref{eqn:depend}, then the FDR can be controlled below level $\alpha$. 

We formally state the results in the following proposition.
\begin{prp} \label{prp:FDR-dependent-general-weight}
	Suppose that the weight distribution satisfies the informative-weighting assumption in~\eqref{weight-condition}. And weighted p-values $(P_t /\omega_t)$ marginally satisfy
	\eqref{eqn:depend}. Then the modified LORD++ rule that applies to the weighted p-values can control the online FDR at any given level $\alpha$, i.e.,
	\beq
	\sup_{T \in \mathbb{N}}\; \text{\em FDR}(T) \le \alpha.
	\eeq
\end{prp}
The proofs of these extension are almost the same as those in \cite{javanmard2018online} and are omitted here.

\paragraph{Lower Bound on Statistical Power of Weighted LORD++.} In order to compare different procedures, it is important to estimate their statistical power. Here, we analyze the power of the weighted LORD++.
Define $D(a) = \Pr[P/\omega \le a]$ as the marginal distribution of weighted p-values. Under the assumptions on the weight distribution from~\eqref{weight}, the marginal distribution of weighted p-value equals,
\begin{align}
D(a)  & = \Pr[P / \omega \le a]  = \int \Pr [P / \omega \le a \mid \omega = w] \, \mathrm{d} Q(w) \nonumber \\
& = \int \sum_{h \in \{0,1\}} \Pr [P / \omega \le a \mid \omega = w, H = h] g(h\mid w)\, \mathrm{d} Q(w) \nonumber \\
& = \int \sum_{h \in \{0,1\}} \Pr[P / w \le a \mid  H = h] g(h\mid w) \, \mathrm{d} Q(w) \nonumber\\
& = \int \sum_{h \in \{0,1\}} ((1-h) aw + h F(aw)) g(h\mid w)\, \mathrm{d} Q(w) \nonumber \\
& = \int \sum_{h \in \{0,1\}} ((1-h) aw + h F(aw))\, \mathrm{d} Q(w\mid h) g(h)  \nonumber\\
& = \sum_{h \in \{0,1\}} \int ((1-h) aw + h F(aw))\, \mathrm{d} Q(w\mid h) g(h)  \nonumber \\
& = (1 - \pi_1) \int aw d Q(w\mid h = 0)+ \pi_1 \int F(aw)\, \mathrm{d} Q(w \mid h = 1) \nonumber \\
& = (1 - \pi_1) \mu_0 a + \pi_1 \int F(aw) \mathrm{d} Q_1(w). \label{eqn:da}
\end{align}

\begin{thm} \label{thm:lowerbound}
Let $D(a) = \Pr[P/\omega \le a]$ be the marginal distribution of weighted p-values as in~\eqref{eqn:da}. Then, the average power of weighted LORD++ rule is almost surely bounded as follows:
\beq 
	\liminf_{T \to \infty} \text{\em TDR}(T) \ge (\sum_{m =1}^{\infty} \prod_{j = 1}^{m}(1-D(b_0 \gamma_j)))^{-1}.
\eeq
\end{thm}
The proof of \thmref{lowerbound} uses the similar technique as the proof of the statistical power of LORD in \citep{javanmard2018online}. The main distinction is that we replace the marginal distribution of p-values by the marginal distribution of weighted p-values. Theorem~\ref{thm:lowerbound} also holds for a weighted LORD procedure (replacing LORD++ by LORD in Definition~\ref{def:weightedlord++}).

\paragraph{Comparison of Power.} Next, we establish conditions under which a weighting could lead to increased power for LORD. We work with (a version of) the popular LORD procedure from~\cite{javanmard2018online}, which sets
\begin{align} \label{simpleLORD}
W(0) = w_0 = b_0 = \alpha/2, \,\, \phi_t = \alpha_t = b_0 \gamma_{t - \tau_t}, \,\, \psi_t = b_0.
\end{align}
As shown by \cite{javanmard2018online}, the average power of LORD, under the mixture model, almost surely equals\!\footnote{\citet{javanmard2018online} proposed multiple versions of LORD, and as noted by them the bound in~\eqref{lbLORD} lower bounds the average power on all the versions of LORD for the above mixture model.}
\begin{align} \label{lbLORD}
\text{For LORD: }\liminf_{T \to \infty} \tdr(T) = (\sum_{m =1}^{\infty} \prod_{j = 1}^{m}(1-G(b_0 \gamma_j)))^{-1},
\end{align}
where $G(a) = (1-\pi_1)U(a) + \pi_1 F(a)$ as defined earlier.

\begin{Def} [Weighted LORD]
Given a sequence of p-values, $(P_1,P_2,\dots)$ and weights $(\omega_1,\omega_2,\dots)$, apply LORD~\eqref{simpleLORD} with level $\alpha$ to the weighted p-values $(P_1/\omega_1, P_2/\omega_2, \dots)$.
\end{Def}
The FDR control of weighted LORD under an independence assumption on p-values and an assumption~\eqref{weight-condition} on weights follows from Proposition~\ref{prp:FDR-general-weight}.
From the proof of Theorem~\ref{thm:lowerbound}, the average power of weighted LORD almost surely equals 
\begin{align} \label{lbLORD++}
\text{For weighted LORD: } \liminf_{T \to \infty} \tdr(T) = \sum_{m =1}^{\infty} \prod_{j = 1}^{m}(1-D(b_0 \gamma_j))^{-1}.
\end{align}

Assume $F$ is differentiable and let $f = F'$ be the PDF of p-values under alternative. Due to the fact that p-values under alternative are stochastically dominated by the uniform distribution, there exists some $a_0 > 0$ such that $f(a) > 1$ for all $0 \le a < a_0$. The following theorem is based on comparing this power on weighted LORD from~\eqref{lbLORD++} with the power on LORD from from~\eqref{lbLORD}.

\begin{thm}\label{thm:comparison} [Power Separation]
Suppose that the parameters in LORD~\eqref{simpleLORD} satisfy $b_0 \gamma_1 < a_0$, and the weight distribution satisfies $\Pr[\omega < a_0/(b_0 \gamma_1 ) \mid H_t = 1] = 1$ for every $t \in \mathbb{N}$ and the informative-weighting assumption in~\eqref{weight-condition}. Then, the average power of weighted LORD is greater than equal to that of LORD almost surely.
\end{thm}
As discussed earlier since the general weighting scheme includes the context-weighting scheme, so the results here indicate that using the informative context-weighting in the LORD rules will help in making more true discoveries.

\begin{remark}
Intuitively, the condition implies that to achieve higher power while controlling FDR, the weights given to alternate hypotheses cannot be too large. Let us discuss this point in the context of weighted LORD++ and context-weighted LORD++. 

In weighted LORD++, we can always assign large weights to make the reweighted p-values small enough to be rejected, in order to achieve high power. But this can lead to a loss in FDR control which is why we need the restriction of $\E[\omega] = 1$ for proving the FDR control in~\prpref{FDR-general-weight}. Therefore, it is natural to have weights not too large.

If we consider reweighting the significance levels as in CwLORD++, assigning a large weight will not affect the FDR control (we prove that FDR is controlled for any choice of weights in~\thmref{fdr-control}). However, the price is paid in terms of power. When we use a large weight, the penalty $\phi_t$ increases and therefore the wealth might go quickly down to zero. Once the wealth is exhausted, the significance levels afterwards must all be zero and thus preventing any further discoveries.
\end{remark}
We now conclude this discussion with a simple example of how the conditions of Theorem~\ref{thm:comparison} are easily satisfied in a common statistical model.

\begin{example}
To further interpret the weight condition in~\thmref{comparison}, let us take a concrete example. Consider the hypotheses $(H_1, \dots, H_T)$ concerning the means of normal distributions (referred to as {\em normal means model}). This model corresponds to getting test statistics $Z_t \sim \cN(\mu, 1)$.  So the two-sided p-values are $P_t = 2 \Phi(-|Z_t|)$, where $\Phi$ is the CDF of standard normal distribution.  Suppose under the null hypothesis $\mu = 0$, and under the alternative hypothesis $ 0< \mu \le4$. Then from simple computation we obtain that $a_0 > 0.022$ for any $\mu$ such that $0 <\mu \le 4$. In fact, $a_0$ increases as $\mu$ decreases. Typically, we set $\alpha = 0.05$, so $b_0 = \alpha/ 2 = 0.025$. Using the sequence of hyperparameters $\{\gamma_t\}_{t \in \mathbb{N}}$ (where $\gamma_t = 0.0722 \log(t \vee 2)/(t \exp(\sqrt{\log t}))$) as suggested by~\citep{javanmard2018online} for this normal means model, we compute that $\gamma_1 \approx 0.117$ when the number of total hypotheses $T = 10^5$. Therefore, we get $a_0/(b_0 \gamma_1) \approx 7.52 > 1$. As long as the weight $\omega$ is less than $7.52$ with probability 1 under the alternative hypotheses (for any $\mu$ such that $0 < \mu < 4$), the condition needed for~\thmref{comparison} is satisfied.
\end{example}

\section{Online FDR Control Experiments using Context-weighted GAI} \label{sec:algorithm}
In this section, we propose a practical procedure for contextual online FDR control based on context-weighted GAI rules, and present numerical experiments to illustrate the performance gains achievable with this procedure. Remember that a context-weighted GAI rule is a contextual GAI rule with
$$\alpha_t(R_1,\dots,R_{t-1},X_1,\dots,X_t) = \alpha_t(R_1, \dots, R_{t-1}) \cdot \omega (X_t; \theta),$$
where $\omega (X_t; \theta)$ is a parametric function (with parameter set $\theta \in \Theta$)  that is used to model the weight function. Technically, we can use any parametric function here. In this paper, we choose a deep neural network (multilayer perceptron) for this modeling due to its expressive power, as noted in a recent batch FDR control result by~\citep{xia2017neuralfdr}.

Given this, a natural goal will be to maximize the number of empirical discoveries while controlling FDR. To do so, we are going to find  $\theta \in \Theta$ that maximizes the number of discoveries (or discovery rate), while empirically controlling the FDR. Note that if the function $\alpha_t (R_1, \dots, R_{t-1})$ is monotone (such as with LORD or LORD++) with respect to the $R_i$'s, the function $\alpha_t(R_1, \dots, R_{t-1}) \cdot \omega (X_t; \theta)$ is also monotone with respect to the $R_i$'s. In particular, this means that under the assumptions on the p-values from Theorem~\ref{thm:fdr-control}, FDR is controlled.

\paragraph{Training the Network, Setting $\theta$.}  Algorithm~\ref{alg:NeuralOnlineFDR} is used for training the network, with the goal of adaptively increasing the discovery rate over time. Given a stream $((P_t,X_t))_{t \in \mathbb{N}}$, the algorithm processes the stream in batches, in a single pass. Let $b \geq 1$ denote the batch size. Let $\theta_j$ be the parameter obtained before batch $j$ is processed, thus $\theta_j$ is only based on all previous p-values and contextual features which are assumed to be independent of all future batches. For each batch, the algorithm fixes the parameters to compute the significance levels for hypothesis in that batch. Define, the Empirical Discovery Rate for batch $j$ as follows:

\beq
\DR_{j} = \frac{\sum_{i = jb+1}^{(j+1)b}\1\{P_i \le \alpha_i(X_i;\theta)\}}{b}.
\eeq
Since the above function is not differentiable, we use the sigmoid function $\sigma$ to approximate the indicator function, and define
\beq
\DR_{j} = \frac{\sum_{i = jb+1}^{(j+1)b} \sigma (\lambda (\alpha_i(X_i;\theta) - P_i))}{b}.
\eeq
Here $\lambda$ is a large positive hyperparameter.
With this, the parameter set $\theta$ can now be optimized by using standard (accelerated) gradient methods in an online fashion. Note that we are only maximizing empirical discovery rate subject to empirical FDR control, and the training does not require any ground truth labels on the hypothesis. In fact, the intuition behind Algorithm~\ref{alg:NeuralOnlineFDR} is very similar to the {\em policy gradient descent} method popular in reinforcement learning, which aims to find the best policy that optimizes the rewards. In an online multiple testing problem, we can regard the number of discoveries as reward and parametric functions as policies. 

%

\begin{algorithm}[H]
	\caption{Online FDR Control with a Context-Weighted GAI Procedure}
	Model the weight function as a multi-layer perceptron (MLP)\;
	\medskip
	\KwIn{A sequence of p-value, contextual feature vector pairs $((P_1,X_1), (P_2,X_2),\dots)$, a monotone GAI rule (such as LORD++) denoted by $\mathbb{G}$ with desired FDR control, batch size $b$, learning rate~$\eta$}
	\KwOut{Neural network model parameter set $\theta \in \Theta$}
	\medskip
	Randomly initialize the parameter set $\theta_0$; batch index $j = 0$\\
	\Repeat{convergence or end-of-stream}{
	\For{$i=1$ to $b$}{
	Consider the pair $(P_{jb+i}, X_{{jb+i}})$  ($i$th hypothesis in the $j$th batch)\\
	Let $\tilde{\alpha} \leftarrow \alpha_{jb+i}(R_1,\dots,R_{jb+i-1})$ (computed as defined by the GAI rule $\mathbb{G}$) \\
	Accept/reject this hypothesis, while satisfying GAI $\mathbb{G}$, with significance level reweighed as $\tilde{\alpha} \cdot \omega(X_{{jb+i}};\theta_j)$	}
	Use the decisions in the $j$th batch to update the empirical discovery proportion ($\DR_{j}$) \\
	Compute the gradient with respect to the parameter set $\frac{\partial\, \DR_j}{\partial\, \theta}$ \\
	Update the parameter set: $\theta_{j+1} \leftarrow \theta_j + \eta \frac{\partial\, \DR_j}{\partial\, \theta}$ \\
	$j \leftarrow j + 1$
	}
	Return $\theta_j$
	\label{alg:NeuralOnlineFDR}
\end{algorithm}

In all our experiments, we use a multilayer perceptron to model the weight function, which is constructed by 10 layers and 10 nodes with ReLU as the activation function in each layer, and exponential function of the output layer, since the weight has to be non-negative. In the following, we use \CwLORDplus (\textbf{CwLORD++}) to denote the testing rule obtained from Algorithm~\ref{alg:NeuralOnlineFDR} by using LORD++ as the monotone GAI rule. Before the experiments, we start with some additional comments about Algorithm~\ref{alg:NeuralOnlineFDR}.

\paragraph{Discussion about Algorithm~\ref{alg:NeuralOnlineFDR}.} A natural question to ask is whether one can check for the informative-weighting assumption~\eqref{weight-condition} in practice. If we assume the feedback (true labels) are given after testing each batch, then this condition can be verified in the online learning process. With the feedback after each batch, we can compute the average weights of the true alternatives and nulls, to see whether the former is greater than the latter. If so, then the weights learned so far are informative and can be utilized further. If not, for next batch we can revert to previous informative weights, or even start over from the baseline unweighted procedure. 

Algorithm~\ref{alg:NeuralOnlineFDR} deals with the case when there is no feedback. The algorithm learns the weight function in an online fashion, by regarding the previous decisions as ground truth, informally meaning that it will regard previously rejected hypotheses as true alternatives and thereby assigning a larger weighting in CwLORD++ for future hypothesis with contextual features similar to those of previously rejected hypotheses. Online FDR control is always guaranteed by Algorithm~\ref{alg:NeuralOnlineFDR}. The hope is that by maximizing the number of empirical discoveries, the algorithm can learn an informative-weighting (as possibly corroborated by our experiments).  However, without feedback, or additional modeling assumptions, it is hard to verify whether the informative-weighting condition is satisfied. Along the same lines, the convergence of Algorithm~\ref{alg:NeuralOnlineFDR} (under a suitable generative model) is a very interesting open problem which will be considered in our future work.

\subsection{Experiments} \label{sec:experiments}
In this section, we present results for numerical experiments with both synthetic and real data to compare the performance of our proposed CwLORD++ testing rule with the current state-of-the-art online testing rule LORD++~\citep{ramdas2017online}. The synthetic data experiments are based on the normal means model, which is commonly used in hypothesis testing literature for comparing the performance of various testing procedures. Our real data experiments focus on a diabetes prediction problem and gene expression data analyses. The primary goal of this section is to show the increased power (under FDR control) in online multiple testing setup that can be obtained with our proposed contextual weighting in many real world problems. Experiments with the SAFFRON procedure are presented in Appendix~\ref{app:saffron}.

\subsubsection{Synthetic Data Experiments} \label{sec:syndata}
For the synthetic data experiments, we consider the hypotheses $\cH (T) = (H_1, \dots, H_T)$ coming from the normal means model. The setup is as follows: for $t \in [T]$, under the null hypothesis, $H_t: \mu_t = 0$, versus under the alternative, $\mu_t = \mu(X_t)$ is a function of $X_t$.  We observe test statistics $Z_t = \mu_t + \eps_t$, where $\eps_t$'s are independent standard normal random variables, and thus the two-sided p-values are $P_t = 2 \Phi(-|Z_t|)$. For simplicity, we consider a linear function $\mu(X_t) = \langle \boldsymbol{\beta}, X_t \rangle$ for $\boldsymbol{\beta}$ unknown to the testing setup. We choose the dimension of the features $X_t$'s as $d = 10$ in all following experiments. 

\begin{figure}[!ht]
\centering
\begin{tabular}{ll}
\begin{subfigure}[b]{0.45\textwidth} 
\centering 
\includegraphics[width=\textwidth]{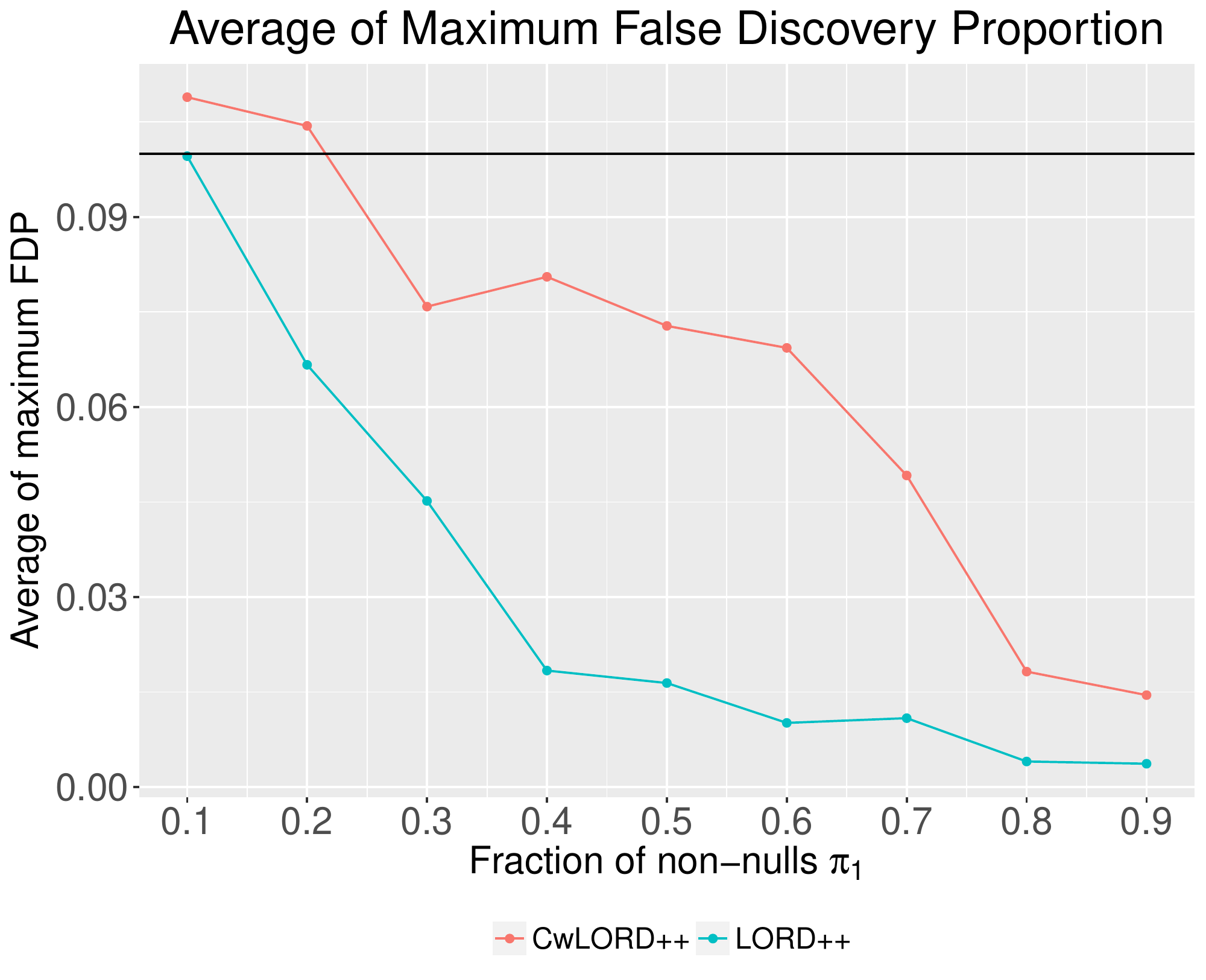} \caption{}
\end{subfigure} & 
\begin{subfigure}[b]{0.45\textwidth} \centering \includegraphics[width=\textwidth]{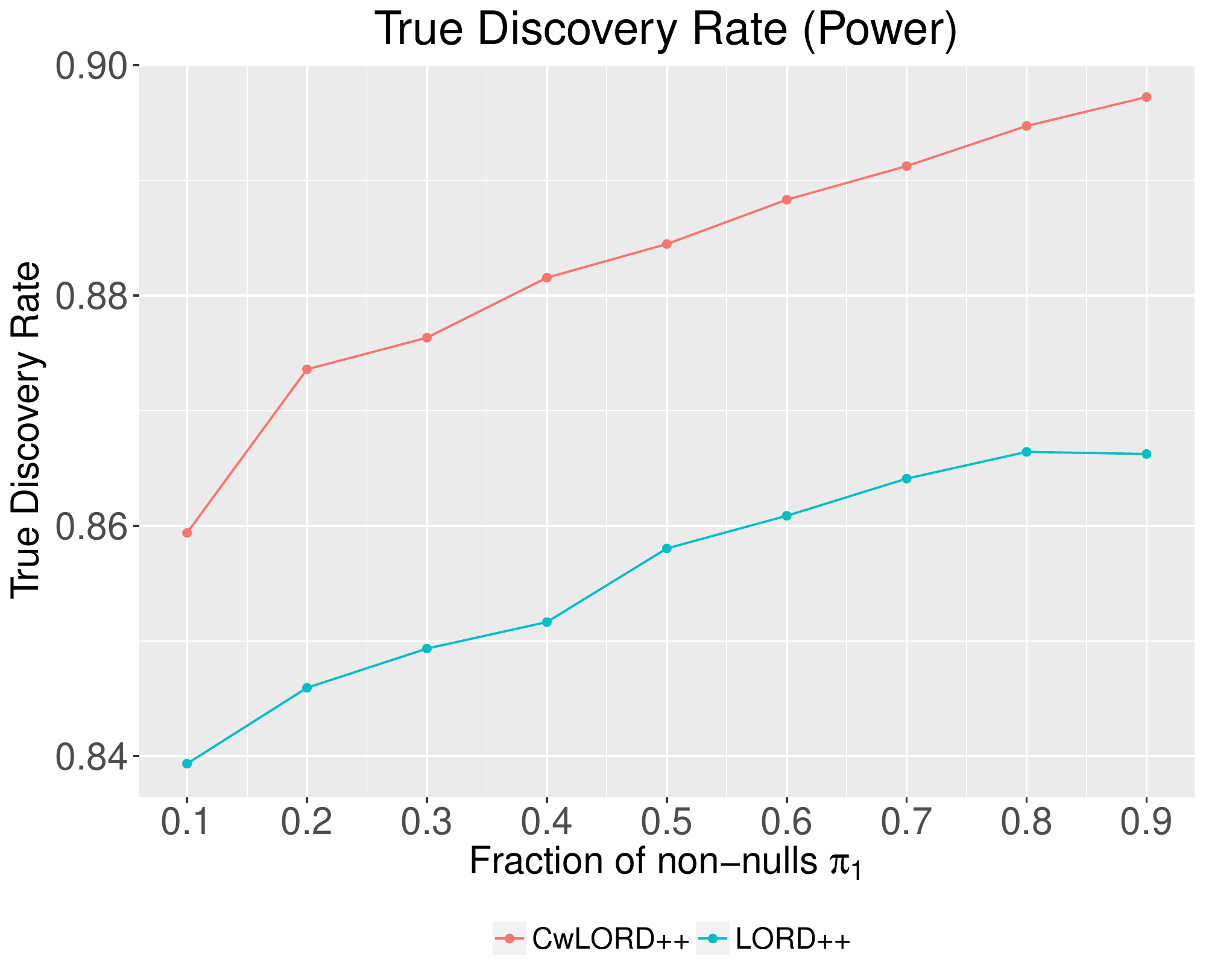}  \caption{} \end{subfigure} \\
\begin{subfigure}[b]{0.45\textwidth}\centering \includegraphics[width=\textwidth]{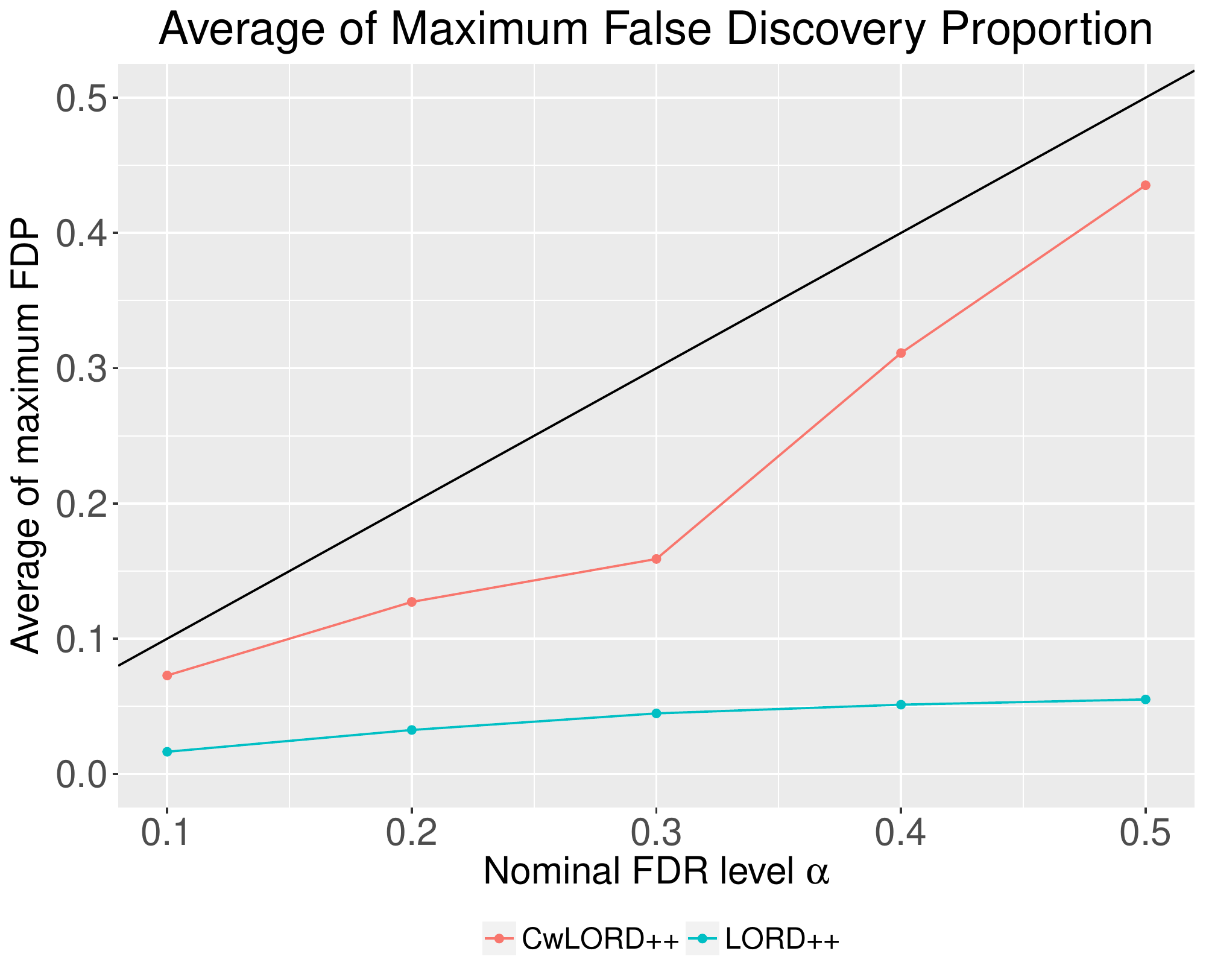}  \caption{} \end{subfigure}  &
\begin{subfigure}[b]{0.45\textwidth} \centering \includegraphics[width=\textwidth]{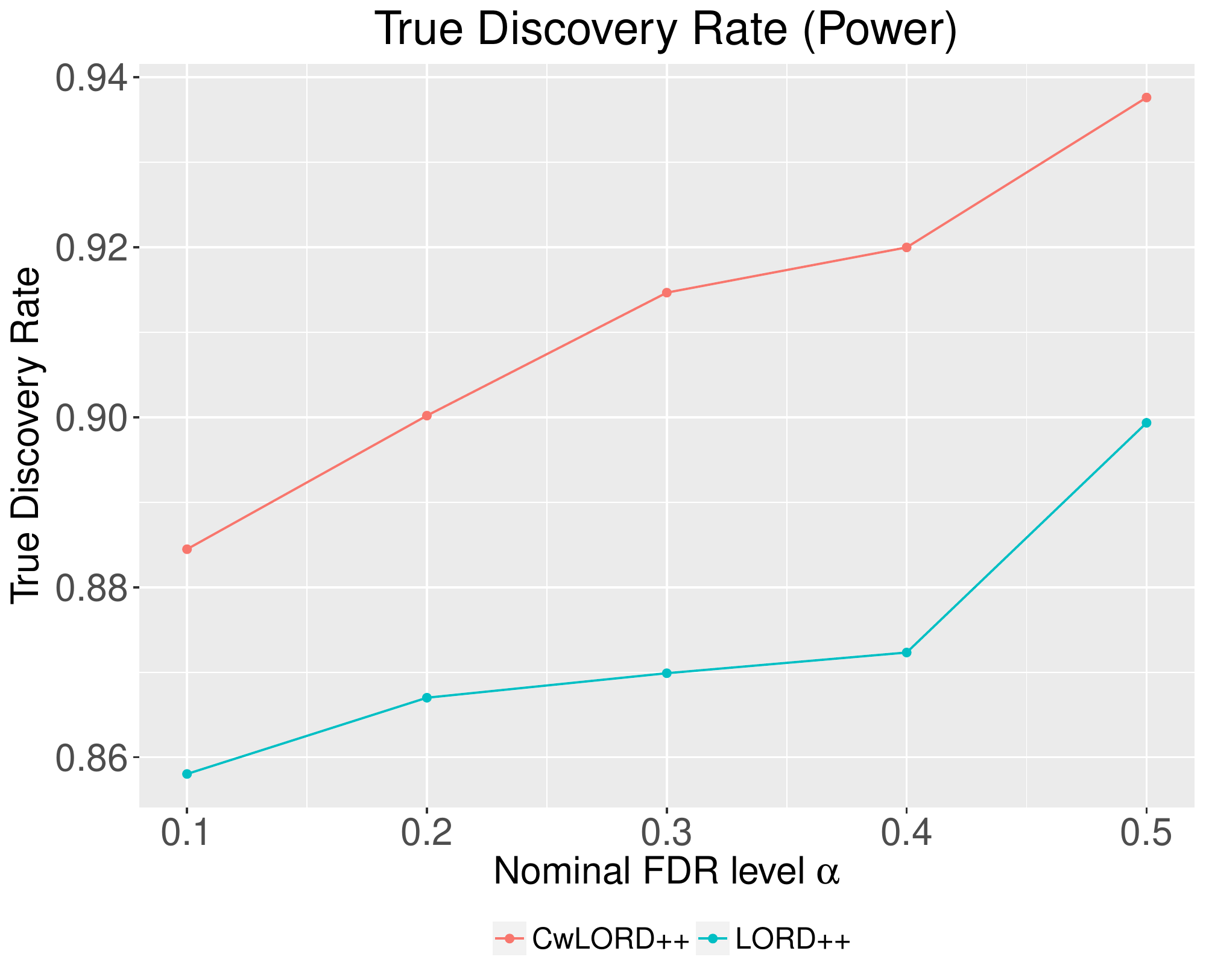} \caption{} \end{subfigure}
\end{tabular}
\caption{The top rows plots the average of max FDP and TDR (power) for our proposed CwLORD++ and LORD++ as we vary the fraction of non-nulls ($\pi_1$) under the normal means model.  The nominal FDR control level $\alpha = 0.1$. The bottom row plots the same with varying nominal FDR levels. In this case, we set the fraction of non-nulls $\pi_1=0.5$. As mentioned in the text, the average of max FDP is an overestimate of FDR. }
\label{fig:fraction}
\end{figure}

We set the total number of hypotheses as $T = 10^5$. We generate each \textit{d}-dimensional vector $X_t$ i.i.d.\ from $\cN (0, \sigma^2 I_d)$ with $\sigma^2 = 2 \log T$. The choice of the $\sigma^2 $ is to put the signals in a detectable (but not easy) region. This is because under the global null hypothesis where $Z_t \sim \cN(0,1) $ for all $t = 1, \dots, T$, we have $\max_{t \in [T]} Z_t \sim \sqrt{2 \log T}$ with high probability.  Here, $\boldsymbol{\beta}$ is a deterministic parameter vector of dimension $d = 10$, we generate the $i$th coordinate in $\boldsymbol{\beta}$  as $\beta_i \sim \text{Uniform}(-2,2)$ and fix $\boldsymbol{\beta}$ throughout the following experiments. Let $\pi_i$ denote the fraction of non-null hypotheses.
For LORD++, we choose the sequence of hyperparameters $\{\gamma_t\}$ (where $\gamma_t = 0.0722 \log(t \vee 2)/(t \exp(\sqrt{\log t}))$) as suggested by~\citep{javanmard2018online}. 

In Figure~\ref{fig:fraction}, we report the maximum FDP and the statistical power  of the two compared procedures as we vary the fraction of non-nulls $\pi_1$ and desired level $\alpha$. The average is taken over 20 repeats. 

In the first set of experiments (Figures~\ref{fig:fraction}(a) and \ref{fig:fraction}(b)), we set  $\alpha = 0.1$, and vary the fraction of non-nulls  $\pi_1$ from $0.1$ to $0.9$. We can see that FDP of both rules (CwLORD++ and LORD++) are almost always under the set level $\alpha = 0.1$ and are decreasing with the increasing fraction of non-nulls. As expected, the power increases with increasing $\pi_1$, however the power of CwLORD++ uniformly dominates that of LORD++.

Note that we take the average of maximum FDP over 20 repeats, which is an estimate for $\E[\sup \fdp]$. Due to the fact that $\E[\sup \fdp] \ge \sup \E[\fdp] = \sup \fdr$, the reported average of maximum FDP is probably higher than the true maximum FDR. In Figure~\ref{fig:fraction}(a), we see that the average of maximum FDP is almost always controlled under the black line, which means the true maximum FDR should be even lower than that level. When $\pi_1$ is really small (like $0.1$), the number of non-nulls is too sparse to make a high proportion of true discoveries, which leads to a higher average of maximum FDP that also have a higher variance in Figure~\ref{fig:fraction}(a). 

In the second set of experiments (Figures~\ref{fig:fraction}(c) and \ref{fig:fraction}(d)), we vary the nominal FDR level $\alpha$ from $0.1$ to $0.5$. The fraction of non-nulls is set as $0.5$. Again we observe while both rules have FDR controlled under nominal level (the black line), and our proposed CwLORD++ is more powerful than the LORD++ with respect to the true discovery rate. On average, we notice about $3$-$5\%$ improvement in the power with CwLORD++ when compared to LORD++. 

\subsubsection{Diabetes Prediction Problem}
In this section, we apply our online multiple testing rules to a real-life application of diabetes prediction. Machine learning algorithms are now commonly used to construct predictive health scores for patients. In this particular problem, we want a test to identify patients that are at risk of developing diabetes.  A high predicted risk score can trigger an intervention (such as medical follow-up, medical tests), which can be expensive and sometimes unnecessary,  and therefore it is important to control the fraction of alerts that are false discoveries. That is, for each patient $i$, we form the null hypothesis $H_i$ as the ``patient will not develop diabetes" versus its alternative. The dataset was released as part of a Kaggle competition\footnote{\url{http://www.kaggle.com/c/pf2012-diabetes}}, which contains de-identified medical records of 9948 patients (labeled as $1,2,\dots$).  For each patient, we have a response variable $Y$ that indicates if the patient is diagnosed with Type 2 diabetes mellitus, along with patient's biographical information and details on medications, lab results, immunizations, 
allergies, and vital signs. 
In the following, we train a predictive score based on the available records, and then will apply our online multiple testing rule rules to control FDR on test set. Our overall methodology is similar to that used by~\citet{javanmard2018online} in their FDR control experiments on this dataset.
We proceed as follows. We construct the following features for each patient. 
\begin{CompactEnumerate}
\item Biographical information: Age, height, weight, BMI (Body Mass Indicator), etc.
\item Medications: We construct TF-IDF vectors from the medication names.
\item Diagnosis information: We derive 20 categories from the ICD-9 codes and construct an one-hot encoded vector.
\item Physician specialty: We categorize the physician specialties and create features that represents how many times a patient visited certain specialist.
\end{CompactEnumerate}
We regard the biographical information of patients as treated as contextual features. The choice of using biographical information as context is loosely based on the idea of {\em personalization} common in machine learning applications.  Note that in theory, for our procedure, one could use other features too as context. 

Second, we split the dataset into four parts \textbf{Train1}, comprising 40\% of the data, \textbf{Train2}, 20\% of the data, \textbf{Test1}, 20\% of the data and \textbf{Test2}, 20\% of the data. The \textbf{Train} sets are used for training a machine learning model (\textbf{Train1}) and for computing the null distribution of test statistics (\textbf{Train2}), which allows us to compute the p-values in the \textbf{Test} sets. We first learn the neural network parameters in the CwLORD++ procedure in an online fashion by applying it to p-values in \textbf{Test1}, and then evaluate the performance of both LORD++ and CwLORD++ on \textbf{Test2}. This process is explained in more detail below.

We note that our experimental setup is not exactly identical to that of \cite{javanmard2018online}, since we are using a slightly different set of features and data cleaning for the logistic regression model. We also split the data to four subsets instead of three as they did, which gives less training data for the predictive model. Our main focus, is to compare the power of LORD++ and CwLORD++, for a reasonable feature set and machine learning model.

\paragraph{Training Process.} We start by training a logistic model similar to~\citep{javanmard2018online}.\!\footnote{Even though the chosen logistic model is one of the best performing models on this dataset in the Kaggle competition, in this paper, we do not actively optimize the prediction model in the training process.} 
Let $\boldsymbol{x_i}$ denote the features of patient $i$. We use all the features to model the probability that patient does not have diabetes through a logistic regression model as
\beq
\Pr[Y_i = 0 \mid \boldsymbol{x} = \boldsymbol{x_i}] = \frac{1}{1 + \exp (\langle \boldsymbol{\beta}, \boldsymbol{x_i} \rangle )}.
\eeq
The parameter $\boldsymbol{\beta}$ is estimated from the \textbf{Train1} set.

\paragraph{Construction of the p-values.}
Let $S_0$ be the subset of patients in \textbf{Train2} set with labels as $Y = 0$, and let $n_0 = |S_0|$. For each $i \in S_0$, we compute its predictive score as $q_i = 1/ (1+ \exp (\langle \boldsymbol{\beta}, \boldsymbol{x_i} \rangle))$.
The empirical distribution of $\{q_i: i \in S_0 \} $ serves as the null distribution of the test statistic, which allows for computation of the p-values. Explicitly, for each $j$ in either \textbf{Test1} or \textbf{Test2} sets, we compute $q_j^{\text{Test}} = 1/ (1+ \exp (\langle \boldsymbol{\beta},\boldsymbol{x_j} \rangle))$, and construct the p-value $P_j$ by 
\beq
P_j = \frac{1}{n_0} \big | \{i \in S_0: q_i \le q_j^{\text{Test}} \} \big |.
\eeq
Smaller p-value indicates that the patient has higher risk of developing diabetes. We use the p-values computed on the patients in \textbf{Test1} to train the weight function in CwLORD++, and the p-values on the patients in \textbf{Test2} to the compare performance of CwLORD++ and LORD++. Note that the training of the neural network does not utilize the labels of the hypothesis in the \textbf{Test1} set. Since the dataset does not have timestamps of hypotheses, we consider an ordering of hypotheses in the ascending order of corresponding p-values, and use this ordering for both LORD++ and CwLORD++. Note that, since the \textbf{Train} and \textbf{Test} sets are exchangeable, the null p-values will be uniform in expectation (and asymptotically uniform under mild conditions).

\paragraph{Online Hypothesis Testing Process and Results.}
We set the desired FDR control level at $\alpha = 0.2$. The set of hyperparameters $\{\gamma_t\}$ is chosen as in the synthetic data experiments.  
For the patients in \textbf{Test1} set, we use their biographical information of patients as contextual features in the training process for CwLORD++ for learning the neural network parameters. We apply the LORD++ and CwLORD++ procedures to the p-values in the \textbf{Test2} set and  compute the false discovery proportion and statistical power. Let $T_2$ be the set of patients in \textbf{Test2} set. Note that for a patient in \textbf{Test2} set, for both CwLORD++ and LORD++, the p-values are identically computed from all the features (including the patient's biographical information). This generates a sequence of p-values $(P_i)_{i \in T_2}$. Now, while LORD++ is applied to this p-value sequence directly, CwLORD++ is applied to the sequence of $(P_i,X_i)$ where $X_i$ is the biographical information of patient $i \in T_2$. For CwLORD++, the neural network parameters are fixed in this testing over the \textbf{Test2} set.

We repeat the whole process and average the results over 30 random splittings of the dataset. Table~\ref{table:diabetes} presents our final result.  We can use biographical information again as contextual features in training CwLORD++ because the p-values under the null are uniformly distributed, no matter which features are used in logistic modeling. This guarantees that the p-values under the null are independent to any features, which is the only condition we need to have a FDR control, assuming the p-values themselves are mutually independent.

Notice that while FDR is under control for both procedures, the statistical power of CwLORD++ is substantially more (about $51\%$)  than LORD++. This improvement illustrates the benefits of using contextual features for improving the power with FDR control in a typical machine learning setup. 
A possible reason for the observed increase in power with CwLORD++ is that in addition to using the labeled data in the \text{Train1} set for training a supervised model, CwLORD++ uses in an unsupervised way (i.e., without considering labels) some features of the data in the \textbf{Test1} set in its online training process with the intent of maximizing discoveries. 

One could also see the effect of training LORD++ on larger dataset. For example, previously we trained LORD++ only on \textbf{Train1} set ($40\%$ of the data) and we completely ignored the \textbf{Test1} set for LORD++. Suppose, we instead train logistic model for LORD++ (but not for CwLORD++) on the union of \textbf{Train1} and \textbf{Test1} set to make the baseline method stronger, with the same splits of the datasets. With this change,   LORD++ has FDR at $0.143$ and power at $0.427$. While this is not completely a fair comparison for CwLORD++ as we have now used more labeled data in training the logistic model for LORD++ than CwLORD++, we observe that  CwLORD++ still beats this stronger baseline with over $35\%$ power increase. 

\begin{table}
\centering
\begin{tabular}{c c c c} 
\hline \hline 
& FDR & Power \\ [0.5ex]
\hline
LORD++ & 0.147 & 0.384 \\
CwLORD++ & 0.176 & 0.580 \\ [0.5ex]
\hline
\end{tabular}	
\caption{Results from diabetes dataset with nominal FDR control level $\alpha = 0.2$.}
\label{table:diabetes}
\end{table}

\begin{figure}[!ht]
\centering \label{figures:diabetes}
\includegraphics[width =  7cm]{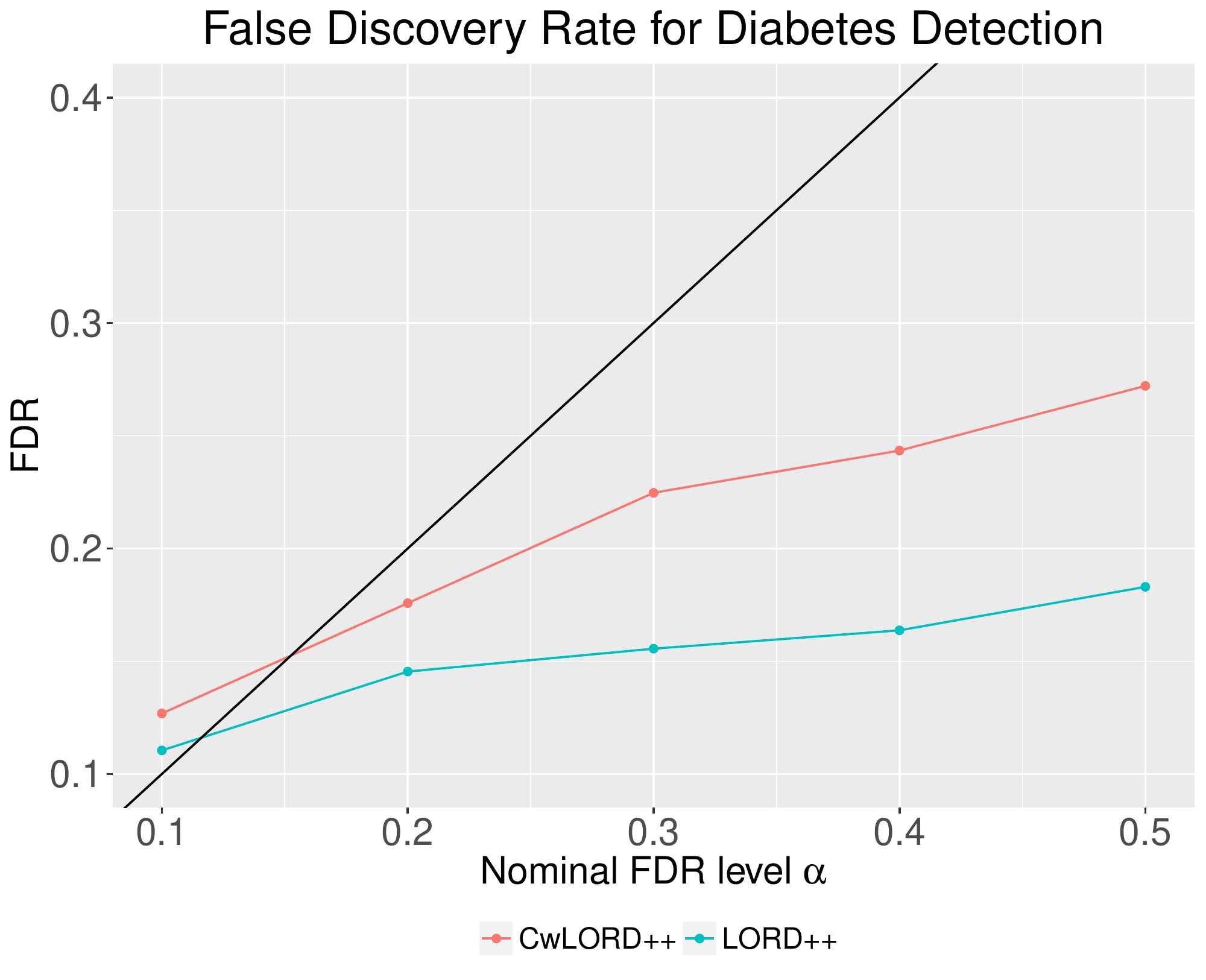}
\includegraphics[width = 7cm]{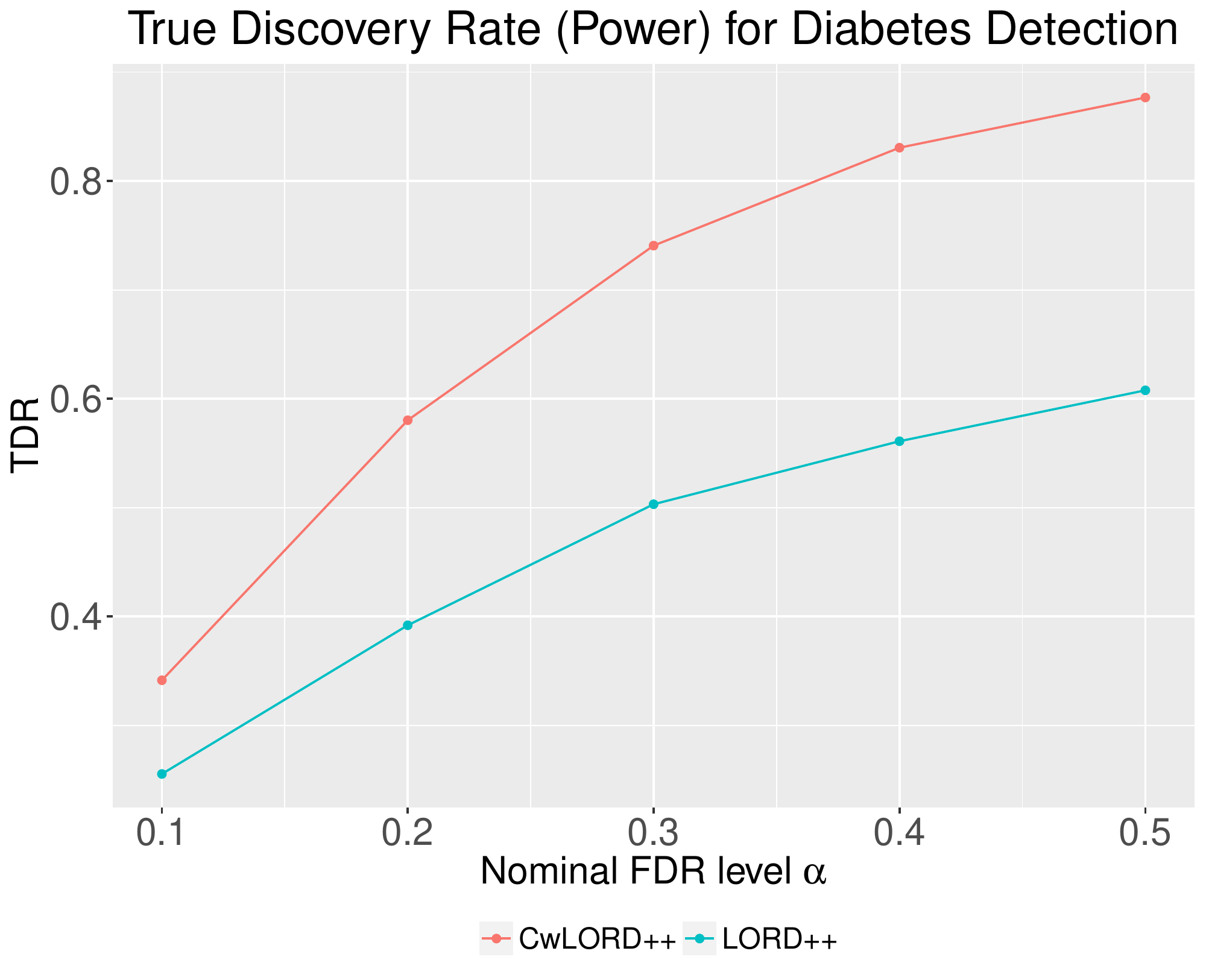}
\caption{FDR and TDR results on diabetes dataset as we vary the nominal FDR level $\alpha$. Note that the power of CwLORD++ uniformly dominates that of LORD++, with an average improvement in power of about $44\%$.}
\label{fig:diabetes}
\end{figure}

In order to further probe some of these improvements, we repeated the experiment with different nominal FDR levels ranging from $0.1$ to $0.5$. The results (see Figure~\ref{fig:diabetes}) demonstrate that our CwLORD++ procedure achieves more true discoveries than the LORD++ procedure while controlling FDR under the same level. The FDR is controlled exactly under the desired level starting around $\alpha \geq 0.15$, while it is close to the desired level even when $\alpha$ is as small as $0.1$. This phenomenon can also be observed in~\citep{javanmard2018online}, where the FDR is $0.126$ for LORD under the target level $\alpha = 0.1$. This is probably because both the experiments here and in~\citep{javanmard2018online} do not adjust for the dependency among the p-values, which violates the theoretical assumption behind the FDR control proof, and is more of a concern when target $\alpha$ level is small.

\subsubsection{Gene Expression Data}
Our final set of experiments are on gene expression datasets. In particular, we use the Airway RNA-Seq  and GTEx datasets\footnote{Datasets are at: \url{https://www.dropbox.com/sh/wtp58wd60980d6b/AAA4wA60ykP-fDfS5BNsNkiGa?dl=0}.} as also studied by~\citep{xia2017neuralfdr}. For both experiments, we use the original ordering of hypotheses as provided in the datasets. Since we don't know the ground truth, we only report the empirical FDR and the empirical discovery rate number in the experiments.

In the Airway RNA-Seq data, the aim is to identify glucocorticoid responsive (GC) genes that modulate cytokine function in airway smooth muscle cells. The dataset contains $n = 33469$ genes. The p-values are obtained in regular two-sample differential analysis of gene expression levels. Log counts of each gene serves as the contextual feature in this case.  \figref{airway} reports the empirical FDR and the discovery number. We see that our CwLORD++ procedure make about $10\%$ more discoveries than the LORD++ procedure.

In the GTEx study, the question is to quantify the expression Quantitative Trait Loci (eQTLs) in human tissues. In the eQTL analysis, the association of each pair of single nucleotide polymorphism (SNP) and nearby gene is tested. The p-value is computed under the null hypothesis that the SNP genotype is not correlated with the gene expression. The GTEx dataset contains 464,636 pairs of SNP-gene combination from chromosome 1 in a brain tissue (interior caudate). Besides the p-values from the correlation test, contextual features may affect whether a SNP is likely to be an eQTL, and thus we can discover more of the true eQTLs if we utilize them in tests. For the tests, we consider the three contextual features studied by~\citep{xia2017neuralfdr}: 1) the distance (GTEx-dist) between the SNP and the gene (measured in log base-pairs); 2)
the average expression (GTEx-exp) of the gene across individuals (measured in log rpkm); and 3) the
evolutionary conservation measured by the standard PhastCons scores (GTEx-PhastCons). We apply LORD++ to the p-values, and CwLORD++ to the p-value, contextual feature vector pairs. \figref{gtex} reports the empirical FDR and the discovery number. In GTEx 1D experiments, we use each contextual feature separately in CwLORD++, which increases the discovery number by $5.5\%$ (using GTEx-dist), $2.6\%$ (using GTEx-exp), $2\%$ (using GTEx-PhastCons) compared to the LORD++ procedure.  In GTEx 2D experiments, GTEx-dist and GTEx-exp  features are used together in CwLORD++ as a two-dimensional vector, and here CwLORD++ procedure make about $6.1\%$ more discoveries than the LORD++ procedure. In GTEx 3D experiments, CwLORD++ uses all the three available features and makes about $6.6\%$ more discoveries than LORD++. These results indicate that additional contextual information could be helpful in making more discoveries.
\begin{figure}[H]
\vspace*{-.5in}
\centering 
	\begin{subfigure}[b]{0.45\textwidth} \centering \includegraphics[width=\textwidth]{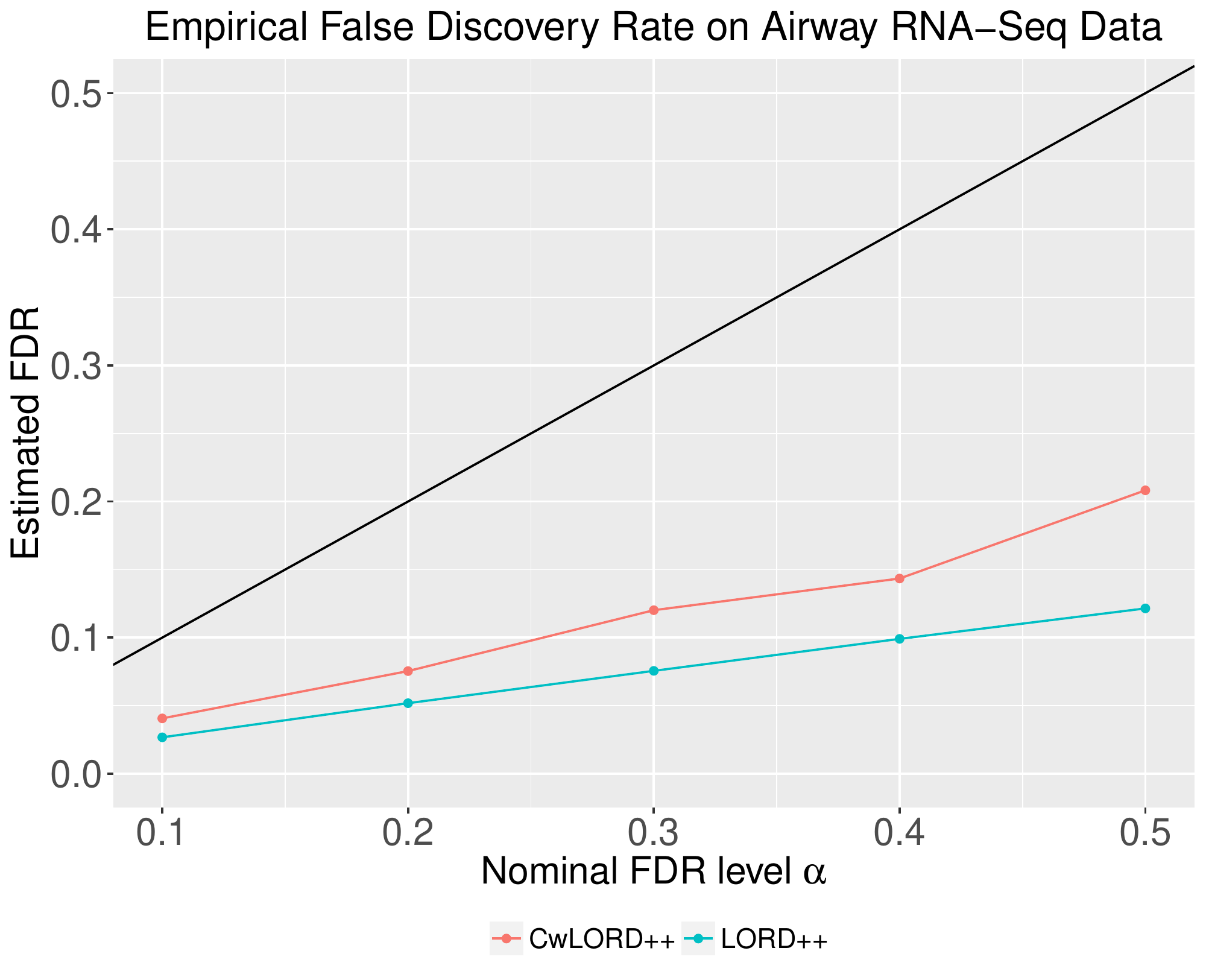} \end{subfigure}
	\begin{subfigure}[b]{0.45\textwidth} \centering \includegraphics[width=\textwidth]{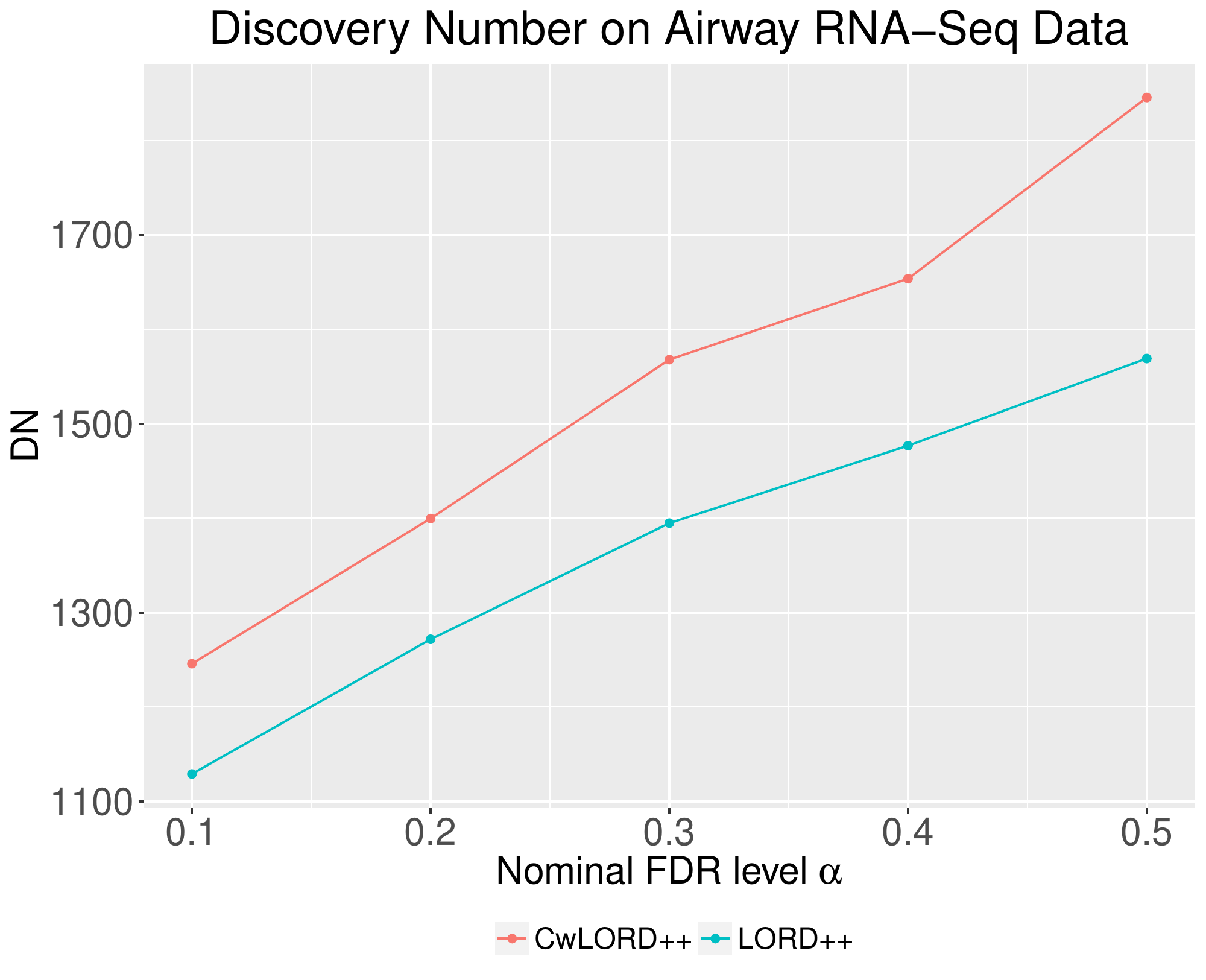} \end{subfigure} 
	\caption{Results on Airway RNA-Seq dataset.}
	\label{fig:airway}
	\vspace*{2ex}
	\begin{subfigure}[b]{0.45\textwidth} \centering \includegraphics[width=\textwidth]{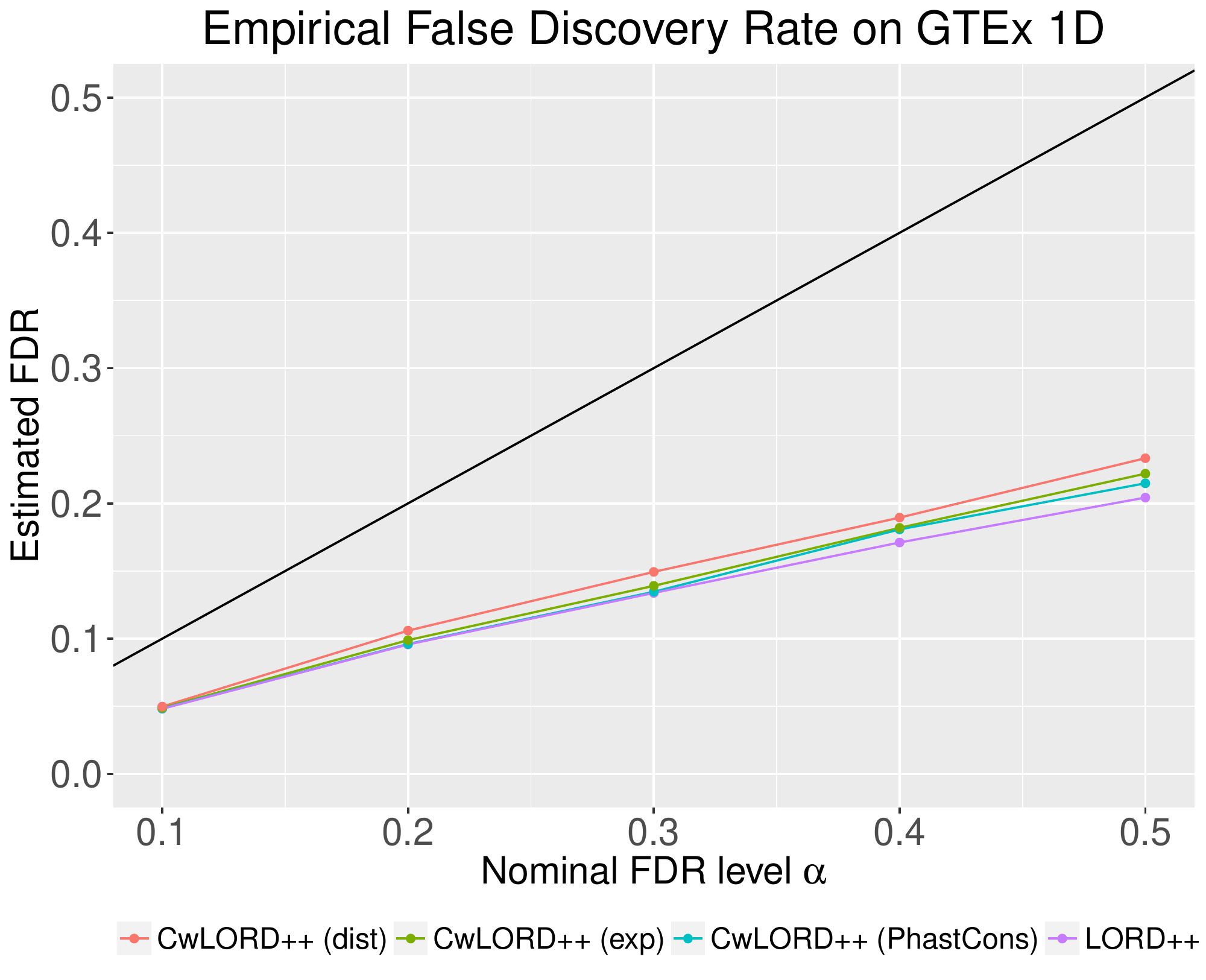} \end{subfigure}
	\begin{subfigure}[b]{0.45\textwidth} \centering \includegraphics[width=\textwidth]{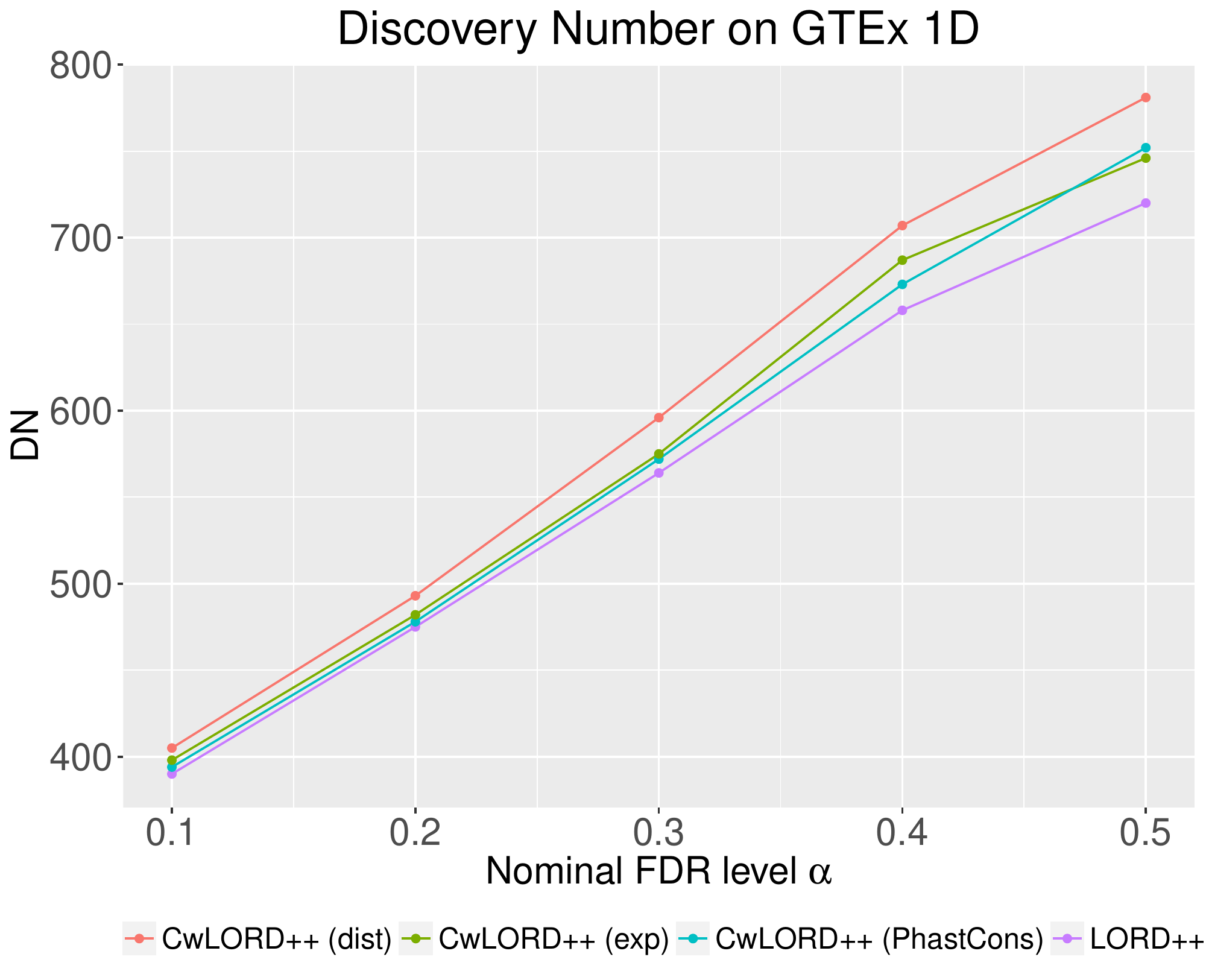} \end{subfigure}
	\begin{subfigure}[b]{0.45\textwidth} \centering \includegraphics[width=\textwidth]{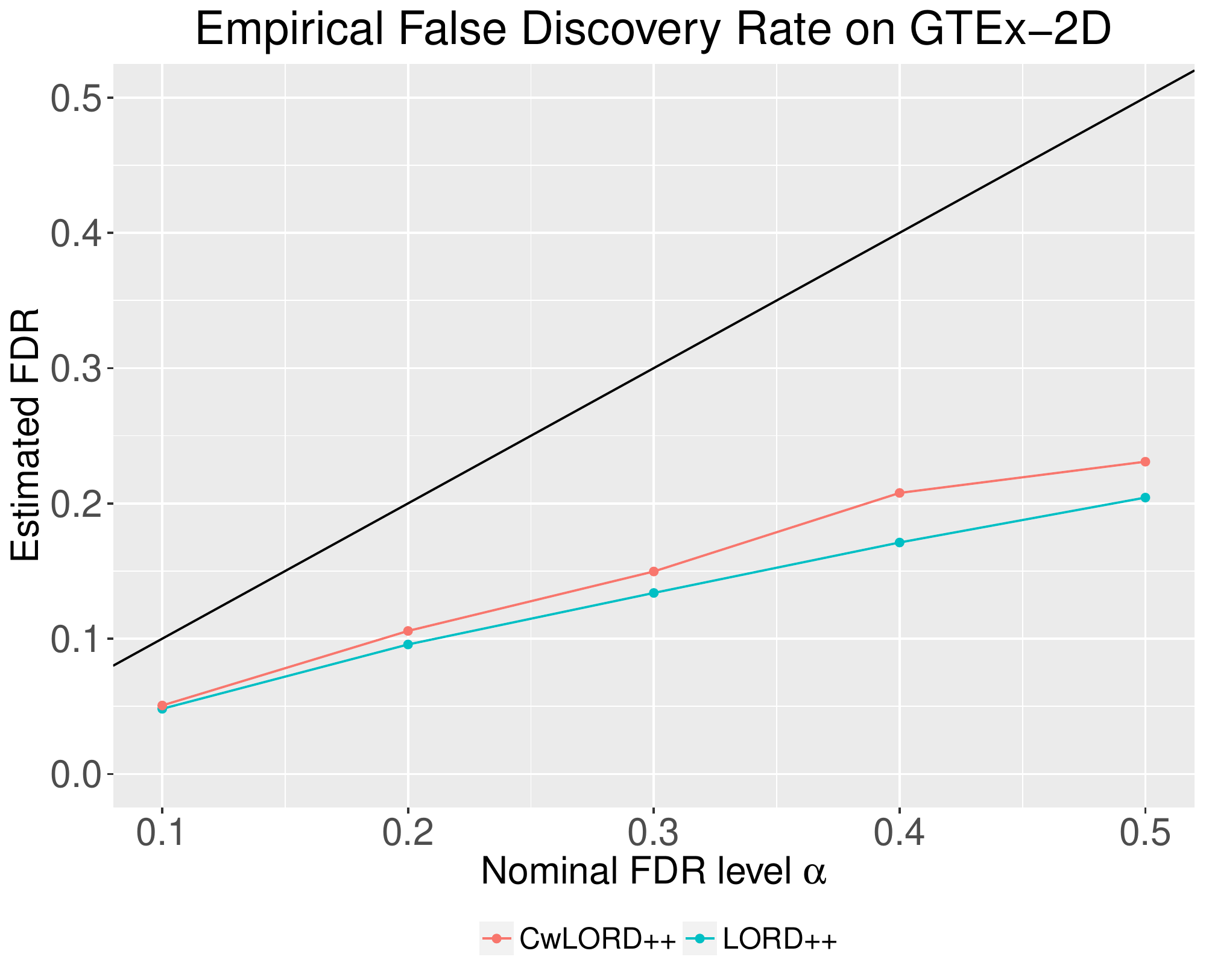} \end{subfigure}
	\begin{subfigure}[b]{0.45\textwidth} \centering \includegraphics[width=\textwidth]{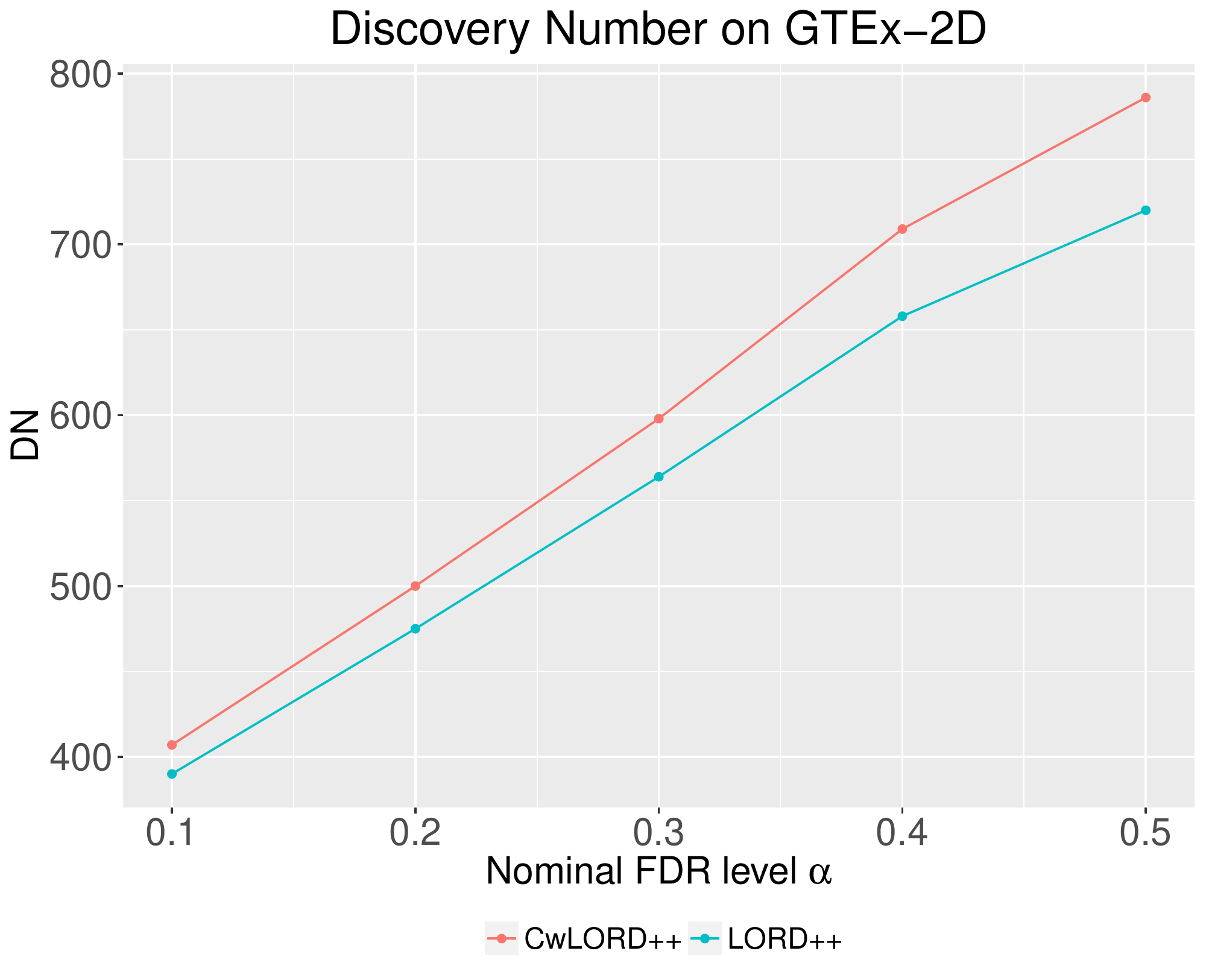} \end{subfigure}
	\begin{subfigure}[b]{0.45\textwidth} \centering \includegraphics[width=\textwidth]{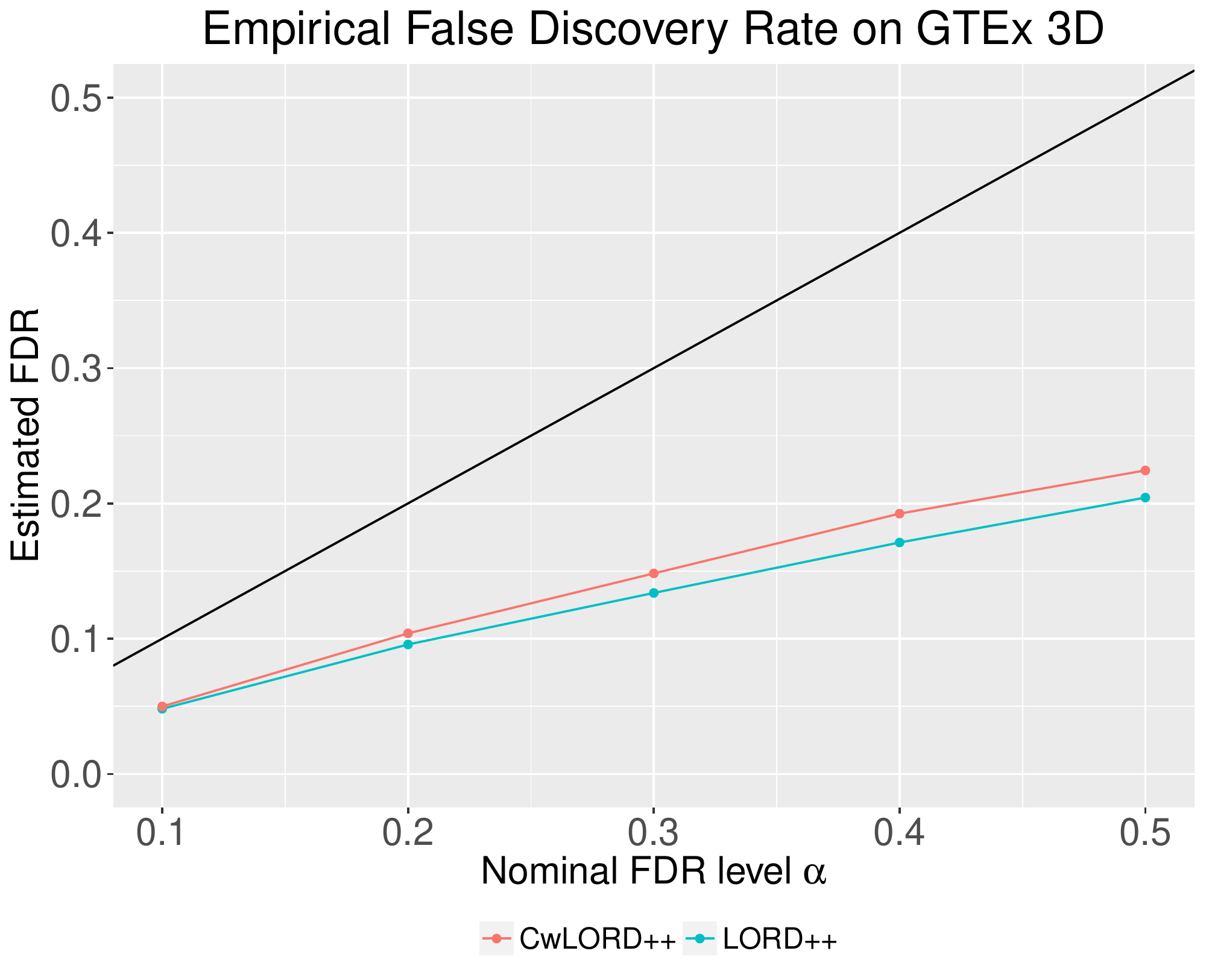} \end{subfigure}
	\begin{subfigure}[b]{0.45\textwidth} \centering \includegraphics[width=\textwidth]{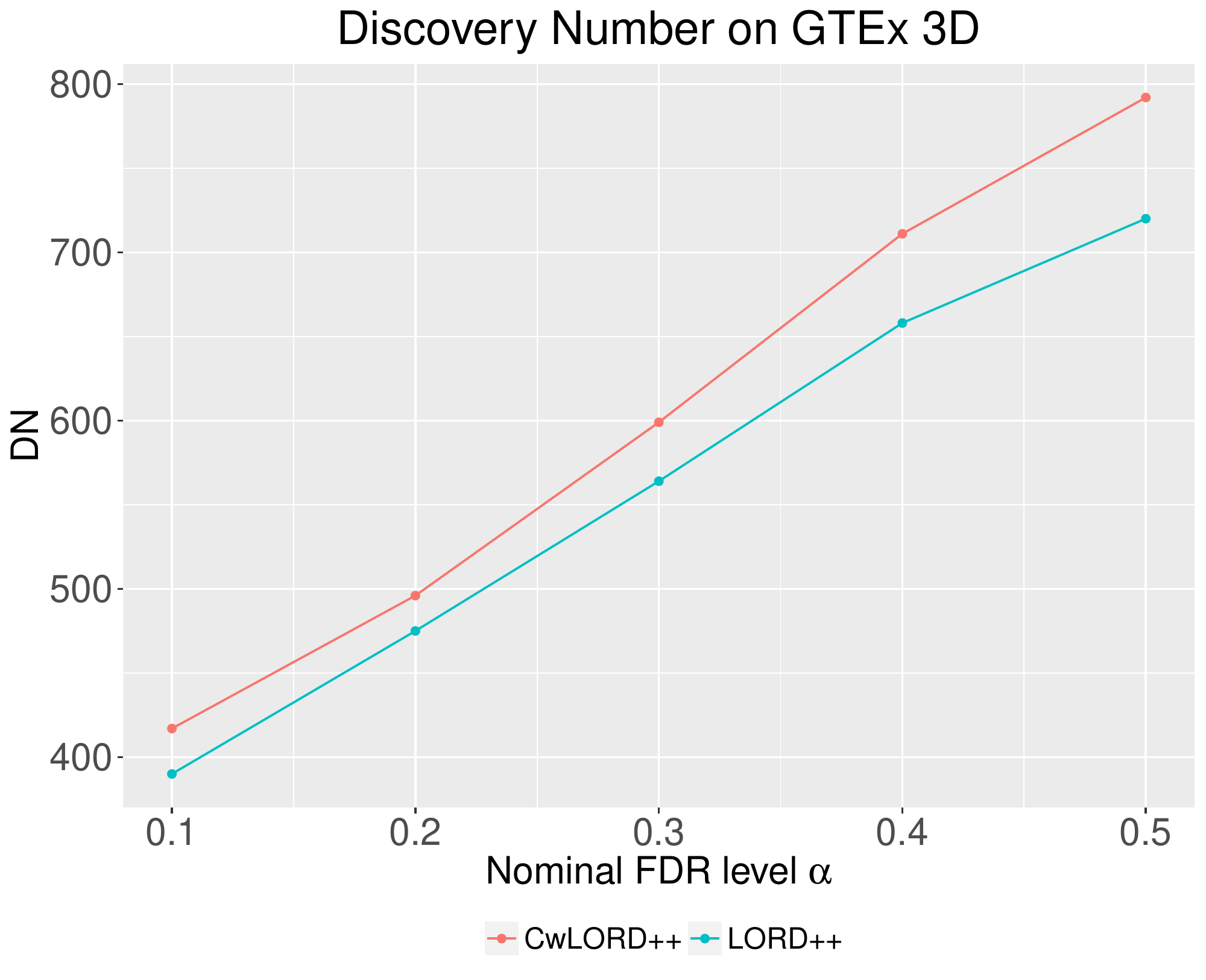} \end{subfigure}
	\caption{Results on GTEx dataset.}
	\label{fig:gtex}
\end{figure}

\section{Conclusion}
In this paper, we introduced a new, rich class of online testing rules which incorporates the available contextual information in the testing process, and can control both the online FDR and mFDR under some standard assumptions. We then focused on a subclass of these rules, based on weighting the significance levels using contextual features, to derive a practical algorithm that learns a parametric weight function in an online fashion to maximize the number of empirical discoveries. We also theoretically proved that under a natural informative-weighting assumption, our procedure can achieve higher statistical power compared to a popular online FDR control procedure, while still controlling the FDR at the same level. Our experiments on both synthetic and real datasets demonstrates the superior performance of our procedure compared to other state-of-the-art online multiple testing procedures.

\section*{Acknowledgements}
We are grateful for the numerous helpful discussions with Dean Foster, David Heckerman, Nina Mishra, and Robert Stine. 
\bibliographystyle{apa-good}
\bibliography{ref}
\appendix

\section{Missing Details from Section~\ref{sec:CWFDR}}\label{app:CWFDR}

\paragraph{Identifiability of $f_1(p\mid X)$.}
We present a simple example from~\citet{lei2018adapt} that illustrates why $f_1(p\mid X)$ (distribution of $p$ under the alternate) is not identifiable.  Consider the following mixture model:
\begin{eqnarray*}
& H_t \mid X_t \stackrel{\text{i.i.d.}}\sim \text{Bernoulli}(\pi_1), & \\
&P_t \mid H_t, X_t  =  \begin{cases} 
	\text{Uniform(0,1)} & \text{if } H_t = 0,\\ 
	f_1(p\mid X_t) & \text{if } H_t = 1.
\end{cases}&
\end{eqnarray*}
Now consider the conditional mixture density $f(p\mid X) = (1-\pi_1)+ \pi_1 f_1(p\mid X)$. Note that the $H_t$'s are not observed. Thus, while $f$ is identifiable from the data, $\pi_1$ and $f_1$ are not: for example, $\pi_1 = 0.5$, $f_1(p \mid  X) = 2(1-p)$ and $\pi_1 = 1$, $f_1(p \mid  X) = 1.5-p$ result in exactly the same mixture density $f(p\mid X)$.

\begin{lem}[\lemref{super-uniform} Restated]
	Let $g: \{0, 1\}^{T} \to \mathbb{R}$ be any coordinatewise non-decreasing function such that $g(\bR) > 0$ for any vector $\bR \neq (0, \dots, 0)$. Then for any index $t \le T$ such that $t \in \cH^0$, we have
	\begin{align*}
		&\E \bigg [\frac{\1 \{P_t \le \alpha_t(R_1, \dots, R_{t-1}, X_1, \dots, X_t)\}}{g(R_1, \dots, R_T) \vee 1}  \bigg \vert \sigma(\cF^{t-1} \cup \cG^t) \bigg ] 
		&\le \E \bigg [ \frac{\alpha_t (R_1, \dots, R_{t-1}, X_1, \dots, X_T)}{g(R_1, \dots, R_T) \vee 1} \bigg \vert \sigma(\cF^{t-1} \cup \cG^t) \bigg ].
	\end{align*}
\end{lem}

\begin{proof}
Let $\bP = (P_1, \dots, P_T)$ be the sequence of p-values, and $\bX = (X_1, \dots, X_T)$ be the sequence of the contextual feature vectors until sometime $T$. We define a ``leave-one-out" vector of p-value as $\widetilde{\bP}^{-t} = (\widetilde{P}_1, \dots, \widetilde{P}_{T})$, which was obtained from $\bP$ by setting $P_t = 0$, i.e.,

\beq 
\widetilde{P}_i = \begin{cases} 
	P_i & \quad \text{if } i \neq t,\\ 
	0 &  \quad \text{if } i = t.
\end{cases}
\eeq
Let $\bR = (R_1, \dots, R_T)$ be the sequence of decisions on the input $\bP$ and $\bX$, and $\widetilde{\bR}^{-t} = (\widetilde{R}_1, \dots, \widetilde{R}_T)$ be the sequence of decisions by applying the same rule on the input $\bP^{-t}$ and $\bX$. Note here we just set one p-value as zero but are not changing the contextual feature vectors.

By the construction of p-values, we have that $R_i  = \widetilde{R}_i$ for $i < t$, and hence
\beq
\alpha_i(R_1, \dots, R_{i-1}, X_1, \dots, X_i) = \alpha_i(\widetilde{R}_1, \dots, \widetilde{R}_{i-1}, X_1, \dots, X_i), \quad \text{for all } i \le t.
\eeq
We also know that $\widetilde{R}_{t} = 1$ always holds due to the fact $\widetilde{P}_t = 0 \le \alpha_t$. Therefore, if the event $\{P_t \le \alpha_t(R_1, \dots, R_{t-1}, X_1, \dots, X_t )\}$ occurs, we have $R_t = \widetilde{R}_{t}$ and thus $\bR = \widetilde{\bR}^{-t}$.

From the above arguments, we conclude that 
\beq
\frac{\1 \{P_t \le \alpha_t(R_1, \dots, R_{t-1}, X_1, \dots, X_t)\}}{g(\bR) \vee 1} = \frac{\1 \{P_t \le \alpha_t(R_1, \dots, R_{t-1}, X_1, \dots, X_t)\}}{g(\widetilde{\bR}^{-t}) \vee 1}.
\eeq

Due to the fact that $t \in \cH^0$ ($H_t=0$), $P_t$ is independent to all contextual features $\bX$ by assumption (as $P_i$'s and $X_i$'s are independent under the null), which gives that $P_t$ is independent of $\sigma(\cF^{t-1} \cup \cG^t)$. And since $P_t$ is independent of $\widetilde{\bR}^{-t}$, we have,
\begin{align}
	 & \E \bigg [\frac{\1 \{P_t \le \alpha_t(R_1, \dots, R_{t-1}, X_1, \dots, X_t)\}}{g(\bR) \vee 1}  \bigg \vert \sigma(\cF^{t-1} \cup \cG^t) \bigg ] \nonumber \\
	& = \E \bigg [\frac{\1 \{P_t \le \alpha_t(R_1, \dots, R_{t-1}, X_1, \dots, X_t)\}}{g(\widetilde{\bR}^{-t}) \vee 1}  \bigg \vert \sigma(\cF^{t-1} \cup \cG^t) \bigg ] \nonumber \\
	& \le \E \bigg [\frac{ \alpha_t(R_1, \dots, R_{t-1}, X_1, \dots, X_t) }{g(\widetilde{\bR}^{-t}) \vee 1}  \bigg \vert \sigma(\cF^{t-1} \cup \cG^t) \bigg ]  \label{ineq1} \\ 
	& \le  \E \bigg [\frac{ \alpha_t(R_1, \dots, R_{t-1}, X_1, \dots, X_t) }{g(\bR) \vee 1}  \bigg \vert \sigma(\cF^{t-1} \cup \cG^t) \bigg ]  \label{ineq2}
\end{align}
where inequality \eqref{ineq1} follows by taking expectation with respect to $P_t$ and using conditional super-uniformity \eqref{eqn:super-uniformity}, and inequality \eqref{ineq2} is derived by the following observation.

Since $\widetilde{P}_t = 0 \le \alpha_t$, we have $\widetilde{R}_{t} = 1 \ge R_t$. Due to the monotonicity of the significance levels, we have 
\beq
\alpha_i(\widetilde{R}_1, \dots, \widetilde{R}_{i-1}, X_1, \dots, X_i ) \ge  \alpha_i(R_1, \dots, R_{i-1}, X_1, \dots, X_i), \quad \text{for all } i > t,
\eeq
ensuring $\widetilde{R}_i \ge R_i$ for all $i $, and thus $g(\widetilde{\bR}^{-t}) \ge g(\bR)$ by the non-decreasing assumption on the function $g$.
\end{proof}

\begin{thm} [\thmref{fdr-control} Restated] \label{thm:proof-fdr}
Consider a sequence of $((P_t,X_t))_{t \in \mathbb{N}}$  of p-values and contextual features. If the p-values $P_t$'s are independent, and additionally $P_t$ are independent of all $(X_t)_{t \in \mathbb{N}}$ under the null $($whenever $H_t=0)$, then for any monotone contextual generalized alpha-investing rule $($i.e., satisfying conditions~\eqref{eq1},~\eqref{eq2},~\eqref{eq3},~\eqref{eq4}, and~\eqref{eq5}$)$, we have online FDR control,
	\beq
	\sup_{T} \; \text{\em FDR}(T) \le \alpha.
	\eeq
\end{thm}
\begin{proof}
Note that the number of false discoveries is $V(T) = \sum_{t = 1}^{T} R_t \1 \{ t \in \cH^0\}$ and the amount of wealth is $W(T) = w_0 + \sum_{t = 1}^{T} (-\phi_t + R_t \psi_t)$.

We can derive the following expression by using the tower property of conditional expectation 
\begin{align}
	\E \bigg [ \frac{V(T) + W(T)}{R(T) \vee 1} \bigg] & = \sum_{t = 1}^{T} \E \bigg[ \frac{R_t \1\{t \in \cH^0 \}+ \frac{w_0}{T} - \phi_t + R_t \psi_t }{R(T) \vee 1} \bigg]  \nonumber \\
	& = \sum_{t = 1}^{T} \E \bigg[ \frac{\frac{w_0}{T} + R_t  (\psi_t + \1\{t \in \cH^0 \}) - \phi_t }{R(T) \vee 1} \bigg] \nonumber \\
	& = \sum_{t = 1}^{T} \E \bigg[ \E \Big [\frac{\frac{w_0}{T} + R_t  (\psi_t + \1\{t \in \cH^0 \}) - \phi_t }{R(T) \vee 1}  \Big \vert \sigma(\cF^{t-1} \cup \cG^t) \Big] \bigg] \label{eqn:comb}
\end{align}
We split the analysis in two cases based on whether $H_t=0$ or $H_t=1$.
\bitem
\item Case 1: Suppose that $t \in \cH^0$.  By applying \lemref{super-uniform}, we have
\begin{align}
	\E \bigg[ \frac{R_t}{R(T) \vee 1} \bigg \vert \sigma(\cF^{t-1} \cup \cG^t) \bigg] & = \E \bigg[ \frac{\1 \{P_t \le \alpha_t\}}{R(T) \vee 1} \bigg \vert \sigma(\cF^{t-1} \cup \cG^t) \bigg] \nonumber  \\
	&\le \E \bigg[ \frac{\alpha_t}{R(T) \vee 1} \bigg \vert \sigma(\cF^{t-1} \cup \cG^t) \bigg]. \label{eqn:expsup}
\end{align}

Since $\psi_t \le \frac{\phi_t}{\alpha_t} + b_t - 1$, we further obtain

\begin{align*}
	 \E \bigg[ \E \Big [\frac{\frac{w_0}{T} + R_t  (\psi_t + \1\{t \in \cH^0 \}) - \phi_t }{R(T) \vee 1}  \Big \vert \sigma(\cF^{t-1} \cup \cG^t) \Big] \bigg] & \le \E \bigg[ \E \Big [\frac{\frac{w_0}{T} + R_t  (\frac{\phi_t}{\alpha_t} + b_t) - \phi_t }{R(T) \vee 1}  \Big \vert \sigma(\cF^{t-1} \cup \cG^t) \Big] \bigg] \\
	& = \E \bigg[ \E \Big [\frac{\frac{w_0}{T} + R_t b_t + \frac{\phi_t}{\alpha_t}(R_t - \alpha_t)}{R(T) \vee 1}  \Big \vert \sigma(\cF^{t-1} \cup \cG^t) \Big] \bigg] \\
	& \le \E \bigg[ \E \Big [\frac{\frac{w_0}{T} + R_t b_t}{R(T) \vee 1}  \Big \vert \sigma(\cF^{t-1} \cup \cG^t) \Big] \bigg],
\end{align*}
where the last inequality follows by applying~\eqref{eqn:expsup}.

\item Case 2: Suppose that $t \not \in \cH^0$. Using the fact that $\psi_t \le \phi_t + b_t$, we have
\begin{align*}
	 \E \bigg[ \E \Big [\frac{\frac{w_0}{T} + R_t  (\psi_t + \1\{t \in \cH^0 \}) - \phi_t }{R(T) \vee 1}  \Big \vert \sigma(\cF^{t-1} \cup \cG^t) \Big] \bigg]  & \le \E \bigg[ \E \Big [\frac{\frac{w_0}{T} + R_t  (\phi_t + b_t) - \phi_t }{R(T) \vee 1}  \Big \vert \sigma(\cF^{t-1} \cup \cG^t) \Big] \bigg] \\
	& =  \E \bigg[ \E \Big [\frac{\frac{w_0}{T} + R_t  b_t + (R_t - 1) \phi_t }{R(T) \vee 1}  \Big \vert \sigma(\cF^{t-1} \cup \cG^t) \Big] \bigg]  \\
	& \le \E \bigg[ \E \Big [\frac{\frac{w_0}{T} + R_t b_t}{R(T) \vee 1}  \Big \vert \sigma(\cF^{t-1} \cup \cG^t) \Big] \bigg].
\end{align*}

\eitem
Combining the bound on $ \E \bigg[ \E \Big [\frac{\frac{w_0}{T} + R_t  (\psi_t + \1\{t \in \cH^0 \}) - \phi_t }{R(T) \vee 1}  \Big \vert \sigma(\cF^{t-1} \cup \cG^t) \Big] \bigg]$ from both cases in~\eqref{eqn:comb} and using the definition of $b_t$, we obtain that,
\begin{align*}
	\E \bigg [ \frac{V(T) + W(T)}{R(T) \vee 1} \bigg] & \le \sum_{t = 1}^{T} \E \bigg[ \frac{\frac{w_0}{T} + R_t  b_t}{R(T) \vee 1} \bigg]  = \E \bigg[\frac{w_0 + \sum_{t= 1}^{T} R_t b_t}{R(T) \vee 1} \bigg] \\
	& \le \E \bigg[\frac{w_0 + \sum_{t= 1}^{T} R_t \alpha - w_0 \1 \{T \ge \rho_1\}}{R(T) \vee 1}\bigg]  = \E \bigg[\frac{w_0 +  \alpha R(T) - w_0 \1 \{T \ge \rho_1\}}{R(T) \vee 1}\bigg]  \le \alpha. 
\end{align*}
This concludes the proof of the theorem.
\end{proof}

\begin{thm} [\thmref{mfdr-control} Restated]
Consider a sequence of $((P_t,X_t))_{t \in \mathbb{N}}$  of p-values and contextual features. If the p-values $P_t$'s are conditionally super-uniform distributed (as in~\eqref{eqn:condsup}) , then for any contextual generalized alpha-investing rule $($i.e., satisfying conditions~\eqref{eq1},~\eqref{eq2},~\eqref{eq3}, and~\eqref{eq4}$)$, we have online mFDR control,
\beq
\sup_{T \in \mathbb{N}} \; \text{\em mFDR}(T) \le \alpha.
\eeq
\end{thm}
\begin{proof}
The conditional super-uniformity implies that under null
\begin{align*}
\E \big[ R_t \big \vert \sigma(\cF^{t-1} \cup \cG^t) \big] \le \alpha_t.
\end{align*}
Now using a proof technique similar to \thmref{proof-fdr}, for any $T \in \mathbb{N}$, we get
\begin{align*}
\E[V(T)] & \le \E[V(T) + W(T)] \\
& = \sum_{t = 1}^{T} \E \bigg[ R_t \1\{t \in \cH^0 \}+ \frac{w_0}{T} - \phi_t + R_t \psi_t \bigg]  \\
& = \sum_{t = 1}^{T} \E \bigg[ \frac{w_0}{T} + R_t  (\psi_t + \1\{t \in \cH^0 \}) - \phi_t \bigg] \\
& = \sum_{t = 1}^{T} \E \bigg[ \E \Big [\frac{w_0}{T} + R_t  (\psi_t + \1\{t \in \cH^0 \}) - \phi_t \Big \vert \sigma(\cF^{t-1} \cup \cG^t) \Big] \bigg] \\
& \le \sum_{t = 1}^{T}  \E \bigg[ \E \Big [\frac{w_0}{T} + R_t b_t \Big \vert \sigma(\cF^{t-1} \cup \cG^t) \Big] \bigg] \\
& = \E \big[w_0 + \sum_{t= 1}^{T} R_t b_t \big] = \E \big[w_0 +  \alpha R(T) - w_0 \1 \{T \ge \rho_1\}\big] \\
& \le \alpha \E \big[R(T) \vee 1 \big],
\end{align*}
where for the second inequality we used an analysis similar to that used in the first case in the proof of~\thmref{proof-fdr}. Therefore, for any $T \in \mathbb{N}$,
\beq
\text{mFDR(T)} = \frac{\E[V(T)]}{\E [R(T) \vee 1 ]} \le \alpha.
\eeq
This concludes the proof of the theorem.
\end{proof}

\section{Missing Details from Section~\ref{sec:power}} \label{app:power}
\begin{prp}[Proposition~\ref{prp:FDR-general-weight} Restated]
Suppose that the weight distribution satisfies the informative-weighting assumption in~\eqref{weight-condition}. Suppose that p-values $P_t$'s are independent, and are conditionally independent of the weights $\omega_t$'s given $H_t$'s. Then the weighted LORD++ rule can control the online FDR at any given level $\alpha$, i.e.,
\beq
\sup_{T \in \mathbb{N}}\; \text{\em FDR}(T) \le \alpha.
\eeq
\end{prp}
\begin{proof}
We start with a frequently used estimator of FDR that is defined as:
\beq
\widehat{\fdp}(T) := \frac{\sum_{t=1}^{T} \alpha_t}{R(T) \vee 1}.
\eeq
As established in Section 4 in \citet{ramdas2017online}, LORD++ applied to any sequence of p-values will ensure that $\sup_{T} \widehat{\fdp}(T) \le \alpha$. We apply LORD++ with the sequence of p-values defined as $\boldsymbol{P'} = (\frac{P_1}{\omega_1}, \frac{P_2}{\omega_2}, \frac{P_3}{\omega_3} \dots)$. Let $P_t' = P_t/\omega_t$ for any $t \in \mathbb{N}$. Then it follows that,
\begin{align} \label{fdr-hat}
\sup_{T \in \mathbb{N}} \widehat{\fdp}(T) = \sup_{T \in \mathbb{N}} \frac{\sum_{t=1}^{T} \alpha_t}{R(T) \vee 1} = \sup_{T \in \mathbb{N}} \frac{\sum_{t=1}^{T} \alpha_t}{(\sum_{t=1}^{T} \1\{P_t' \le \alpha_t \} ) \vee 1}\le \alpha.
\end{align}

 We denote the sigma-field of decisions based on the weighted p-values $\boldsymbol{P'}$ till time $t$ as $\cC^{t} = \sigma (R_1, \dots, R_t)$. By using the ``leave-one-out" method used in~\thmref{proof-fdr}, the FDR of the weighted LORD++ at any time $T$ can be written as,
\begin{align*}
\fdr(T) & = \E \bigg[ \frac{\sum_{t=1}^{T} \1 \{t \in \cH^0: P_t' \le \alpha_t \}}{ (\sum_{t=1}^{T} \1\{P_t' \le \alpha_t \} ) \vee 1}\bigg] \\\
& = \sum_{t=1}^{T}\E \bigg[ \frac{\1 \{t \in \cH^0: \frac{P_t}{\omega_t} \le \alpha_t \}}{R(T) \vee 1}\bigg] \\
& = \sum_{t=1}^{T} \E \bigg[ \E \Big[ \frac{\1 \{t \in \cH^0: \frac{P_t}{\omega_t} \le \alpha_t \}}{R(T) \vee 1} \Big \vert \cC^{t-1} \Big]\bigg] \\
& = \sum_{t=1}^{T} \E \bigg[ \E \Big[ \frac{\1 \{t \in \cH^0: \frac{P_t}{\omega_t} \le \alpha_t \}}{R^{-t}(T) \vee 1} \Big \vert \cC^{t-1} \Big]\bigg],
\end{align*}
where $R^{-t}(T) =\sum_{i=1}^{T} \1\{P_i/\omega_i \le \alpha_i\}$ is obtained by setting $P_t = 0$, while keeping all $\omega_t$'s unchanged. 
The last equality holds due to the fact that $R^{-t}(T) = R(T)$ given the event $ \{P_t /\omega \le \alpha_t \}$. 

Since $\alpha_t \in \cC^{t-1}$, and $P_t, \omega_t$ are independent of $R^{-t}(T)$ and $\cC^{t-1}$, we can take the expectation of the numerator inside the brackets and obtain that 
\begin{align*}
\Pr[P_t/\omega_t \le \alpha_t \mid \cC^{t-1}, H_t = 0] & = \int \Pr [P_t / \omega_t \le \alpha_t \mid \cC^{t-1}, \omega_t = w, H_t = 0] \, \mathrm{d} Q(w \mid  H_t = 0) \\
& = \int w \alpha_t d Q(w\mid H_t = 0) \\
& = u_0 \alpha_t,
\end{align*}
where $u_0 = \E[\omega \mid H_t = 0]$. Plugging this in the bound on $\fdr(T)$ from above gives,
\begin{align}
\fdr(T) & = \sum_{t=1}^{T} \E \bigg[ \E \Big[\frac{u_0 \alpha_t}{R^{-t}(T) \vee 1} \Big \vert \cC^{t-1} \Big]\bigg] \nonumber \\
& \le \sum_{t=1}^{T} \E \bigg[ \E \Big[\frac{\alpha_t}{R^{-t}(T) \vee 1} \Big \vert \cC^{t-1} \Big]\bigg] \label{mean_of_weight} \\
& \le \sum_{t=1}^{T} \E \bigg[ \E \Big[\frac{\alpha_t}{R(T) \vee 1} \Big \vert \cC^{t-1} \Big]\bigg] \label{leave-one}\\
& = \E \bigg[ \frac{\sum_{t=1}^{T} \alpha_t}{R(T) \vee 1}\bigg] = \E \big[  \widehat{\fdp}(T) \big] \le \alpha, \nonumber
\end{align}
where inequality \eqref{mean_of_weight} is due to the assumption that $u_0 < 1$,~\eqref{leave-one} follows by the fact that $R^{-t}(T) \ge R(T)$ due to monotonicity of LORD++, and the last equality is based on~\eqref{fdr-hat}.
\end{proof}

\begin{thm} [Theorem~\ref{thm:lowerbound} Restated]
Let $D(a) = \Pr[P/\omega \le a]$ be the above marginal distribution of weighted p-values. Then, the average power of weighted LORD++ rule is almost surely bounded as follows:
\beq 
	\liminf_{T \to \infty} \text{\em TDR}(T) \ge (\sum_{m =1}^{\infty} \prod_{j = 1}^{m}(1-D(b_0 \gamma_j)))^{-1}.
\eeq
\end{thm}
\begin{proof}
Since we are interested in lower bounds, we consider a version of LORD (as also considered in~\citep{javanmard2018online}) which that is based on the following rule,
\begin{align*}
\textbf{LORD:} W(0) = w_0 = b_0 =\alpha/2,  \quad \phi_t = \alpha_t = b_0 \gamma_{t - \tau_t},  \quad \psi_t = b_0.
\end{align*} 
Note that since $b_t = \alpha - w_0 \1 \{\rho_1 > t-1\} >  b_0$ in LORD++, the test level in LORD++ is at least as large to the test level in LORD. Therefore, for any p-value sequence the power of the LORD from is also a lower bound on the power of LORD++. In the rest of this proof, we focus on LORD for the weighted p-value sequence $\{P_1/\omega_1,P_2/\omega_2,\dots\}$. The bound established below is in fact {\em tight} for LORD under this p-value sequence.

Denote by $\rho_i$ as the time of the $i$th discovery (rejection), with $\rho_0 = 0$, and $\Delta_i = \rho_i - \rho_{i-1}$ as the $i$th time interval between the $(i-1)$st and $i$th discoveries.
Let $r_i := \1 \{\rho_i \in \cH^1 \}$ be the reward associated with inter-discovery $\Delta_i$. 
Since the weighted p-values are i.i.d.\ it can be seen that the times between successive discoveries are i.i.d.\ according to the testing procedure LORD, and the process $R(T) = \sum_{l = 1}^{T} R_l$ is a {\em renewal process}~\citep{cox1967renewal}. In fact, for each $i$, we have
\begin{align*}
	\Pr [\Delta_i \ge m] & = \Pr [\cap_{l = \rho_{i-1}}^{\rho_{i-1} + m} \{P_l /  \omega_l > \alpha_l \} ] = \prod_{l = \rho_{i-1}}^{\rho_{i-1} + m} (1 - D(\alpha_l)) = \prod_{l = \rho_{i-1}}^{\rho_{i-1} + m} (1 - D(b_0 \gamma_{l - \rho_{i-1}})) = \prod_{l = 1}^{m} ( 1- D(b_0 \gamma_l)).
\end{align*}
The above expression is same for every $i$. Therefore, 
\beq
\E [\Delta_i] = \sum_{m = 1}^{\infty} \Pr [\Delta_{i} \ge m] = \sum_{m = 1}^{\infty} \prod_{l = 1}^{m} (1 - D(b_0 \gamma_l)).
\eeq

Applying the strong law of large numbers for renewal-reward processes~\citep{cox1967renewal}, we obtain that the following statement holds almost surely,
\beq
\lim_{T \to \infty} \frac{1}{T} \sum_{i = 1}^{R(T)} r_i = \frac{\E(r_i)}{\E(\Delta_1)} = \pi_1 (\sum_{m = 1}^{\infty} \prod_{l = 1}^{m} (1- D(b_0 \gamma_l))) ^{-1}.
\eeq
Let $|\cH^1(T)|$ be the number of true alternatives till time $T$. Since $\lim_{T \to \infty} |\cH^1(T)|/T = \pi_1$ almost surely, we have
\beq
\lim_{T \to \infty} \frac{1}{|\cH^1(T)|} \sum_{i \in \cH^1(T)} R_i = \lim_{T \to \infty} \frac{1}{|\cH^1(T)|} \sum_{i = 1}^{R(T)} r_i = (\sum_{m = 1}^{\infty} \prod_{l = 1}^{m} (1- D(b_0 \gamma_l))) ^{-1}.
\eeq
Now by using the definition of $\text{TDP}(T)$, almost surely, we have that for any weighted LORD++,
$$\liminf_{T \to \infty} \tdp (T) \ge (\sum_{m =1}^{\infty} \prod_{j = 1}^{m}(1-D(b_0 \gamma_j)))^{-1}.$$
As discussed above, this bound translates into a lower bound for weighted LORD++. 
Furthermore, by using the Fatou's lemma~\citep{carothers2000real}, we can extend the same result for $\text{TDR}(T)$ almost surely,
$$\liminf_{T \to \infty} \tdr (T) = \liminf_{T \to \infty} \E [\tdp (T)] \ge \E [\liminf_{T \to \infty} \tdp (T)] \ge (\sum_{m =1}^{\infty} \prod_{j = 1}^{m}(1-D(b_0 \gamma_j)))^{-1}.$$
\end{proof}

\begin{thm}[Theorem~\ref{thm:comparison} Restated]
Suppose that the parameters in LORD~\eqref{simpleLORD} satisfy $b_0 \gamma_1 < a_0$, and the weight distribution satisfies $\Pr[\omega < a_0/(b_0 \gamma_1 ) \mid H_t = 1] = 1$ for every $t \in \mathbb{N}$ and the informative-weighting assumption in~\eqref{weight-condition}. Then, the average power of weighted LORD is greater than equal to that of LORD almost surely.
\end{thm}
\begin{proof}
We compare the average power bound of weighted LORD and LORD. It is equivalent to comparing $D(a)$ and $G(a)$ for $a = b_0 \gamma_l$, for $l = 1, \dots, \infty$. Since $u = (1-\pi_1) u_0 + \pi_1 u_1 = 1$, we have $(1-\pi_1) u_0 = 1- \pi_1 u_1$. This means that
\begin{align*}
	D(a) - G(a) & = (1 - \pi_1) u_0 a + \pi_1 \int F(aw)\, \mathrm{d} Q_1(w) - (1 - \pi_1)a - \pi_1 F(a) \\
	& =  (1 - \pi_1)(u_0 -1) a + \pi_1 (\int F(aw)\, \mathrm{d} Q_1(w) - F(a)) \\
	& = \pi_1 (1-u_1) a + \pi_1 (\int F(aw)\, \mathrm{d} Q_1(w) - F(a)).
\end{align*}
So we just need to compare $(\mu_1 - 1) a$ and $\int F(aw)\, \mathrm{d} Q_1(w) - F(a)$, for any $a =  b_0 \gamma_l$, for $l = 1, \dots, \infty$.
Due to the fact that $\{\gamma_l\}$ is a non-increasing sequence, we have $a = b_0 \gamma_l \le b_0 \gamma_1$.
Since $b_0 \gamma_1 < a_0$ and $\Pr[\omega < a_0/(b_0 \gamma_1 ) \mid H = 1] = 1$ by assumption, then $\Pr [\max(a, aw) <a_0 \mid H = 1] = 1$.

For any fixed $a = b_0 \gamma_l > 0$, we have 
\begin{align}
	\frac{\int F(aw)\, \mathrm{d} Q_1(w) - F(a)}{a} & = \int \frac{F(aw) - F(a)}{a}\, \mathrm{d} Q_1(w) \nonumber \\
	& = \int \frac{F(aw) - F(a)}{(w-1)a} (w-1)\, \mathrm{d} Q_1(w)  \nonumber\\
	& = \int f(\xi) (w-1)\, \mathrm{d} Q_1(w) \label{intermediate} \\ 
	& \ge \int (w-1)\, \mathrm{d} Q_1(w) \label{density} \\
	& = \E[W \mid H = 1] - 1 = u_1 -1,\nonumber
\end{align}
for some $\xi \in (\min(a, aw), \max(a, aw))$. Note we assume $Q_1$ is a continuous distribution, so $\Pr[w = 1 \mid H=1]=0$.
The equality \eqref{intermediate} is achieved by applying the Intermediate Value Theorem, and the inequality \eqref{density} is obtained by the fact that $\Pr[\xi < a_0 \mid H = 1] = 1$, i.e., $\Pr[f(\xi) > 1 \mid H = 1] = 1$.

Therefore, we prove that $\int F(aw)\, \mathrm{d} Q_1(w) - F(a) \ge u_1 - 1$, which implies that $D(a) \ge G(a)$ for $a = b_0 \gamma_l$, for $l = 1, \dots, \infty$.
\end{proof}

\section{SAFFRON Procedure}\label{app:saffron}
Let us start with a quick introduction to the SAFFRON procedure proposed by~\cite{ramdas2018saffron}. Since SAFFRON can be viewed as an online analogue of the famous offline Storey-BH adaptive procedure~\citep{storey2002direct}, we start a description of the Storey-BH procedure.

In the offline setting where p-values are all available, the rejection rule is to reject all p-values below some threshold $s$, meaning that $\cR(s) =\{i \mid P_i \le s\}$. Thus an oracle estimate for FDP is given by
\beq
\fdp^*(s):=\frac{|\cH^0|\cdot s}{|\cR(s)| \vee 1}.
\eeq
The world oracle means that $\fdp^*$ cannot be calculated, since $\cH^0$ is unknown. The BH method overestimates $\fdp^*(s)$ by the empirically computable quantity
\beq
\widehat{\fdp}_{\text{BH}}(s):=\frac{n\cdot s}{|\cR(s)|\vee 1},
\eeq
and chooses the threshold $\hat{s}_{\text{BH}}(s) = \max \{s: \widehat{\fdp}_{\text{BH}}(s) \le \alpha \}$. 
However, \cite{storey2002direct} noted that the estimate of $\widehat{\fdp}_{\text{BH}}(s)$ is conservative, and thus proposed a different estimate (referred to as Storey-BH) as
\beq
\widehat{\fdp}_{\text{St-BH}}(s):=\frac{n\cdot s \cdot \hat{\pi}_0}{|\cR(s)|\vee 1},
\eeq
where the fraction of nulls $\hat{\pi}_0$ is estimated by 
\beq
\hat{\pi}_0:= \frac{1 + \sum_{i = 1}^{n} \1 \{P_i > \lambda \} }{n(1-\lambda)},
\eeq
with a well-chosen $\lambda$. There is a bias-variance trade-off in
the choice of $\lambda$. When $\lambda$ grows larger, the bias of $\hat{\pi}_0$ grows smaller while the variance becomes larger.
Through numerical simulations~\citet{storey2002direct} demonstrated that there could be an increase in power (over the BH procedure) with this adaptivity.

Similarly, in the online setting, the oracle FDP estimate now is 
\beq
\fdp^*(T) := \frac{\sum_{t \in [T], t \in \cH^0} \alpha_t}{R(T) \vee 1}
\eeq
The connection between SAFFRON and LORD/LORD++ is the same as that between Storey-BH and BH. Empirically, LORD/LORD++ overestimates the oracle $\fdp^*(T)$ as 
\beq
\widehat{\fdp}_{\text{LORD}}(T) := \frac{\sum_{t \in [T]} \alpha_t}{R(T) \vee 1}
\eeq
SAFFRON estimates the amount of alpha-wealth that was spent testing nulls so far, which is analogous to the proportion of nulls in the offline setting, and controls the following overestimate of oracle FDP,
\beq
\widehat{\fdp}_{\text{SAFFRON}}(T) := \frac{\sum_{t \in [T]} \alpha_t  \frac{\1 \{P_t > \lambda_t \} }{(1-\lambda_t)} }{R(T) \vee 1},
\eeq
where $\{\lambda_t\}_{t = 1}^{\infty}$ is predictable sequence of user-chosen parameters in interval $(0,1)$. Note that when $\lambda_t = 0$, it recovers $\widehat{\fdr}_{\text{LORD}}(T)$. \cite{ramdas2018saffron} proved that, under some constraints, SAFFRON can control online FDR at given level $\alpha$ by showing that 
$$\fdr(T) \le \E [\widehat{\fdp}_{\text{SAFFRON}}(T)].$$

However, when we extend SAFFRON to incorporate contextual features, due to the optimization problem that we solve in Algorithm~\ref{alg:NeuralOnlineFDR}, it seems difficult to ensure FDR control in practice. The possible reason here is analogous to overfitting in machine learning, due to the bias-variance tradeoff inherent in the SAFFRON estimator. When $\lambda_t$'s is set large, the bias of the estimator becomes smaller but then the variance increases. And when we optimize an FDR estimator with large variance in the training process (meaning the precision of FDR estimator is poor), we will lose FDR control in the validation process, just like overfitting.

We still present some empirical evidence that contextual information could help with SAFFRON too, though one has to be careful about FDR control.

\paragraph{Experiments with SAFFRON.}  We consider exactly the same setting as with the synthetic data experiments in~\secref{syndata}, and train a context-weighted SAFFRON (referred to as CwSAFFRON) in the same way as CwLORD++. With varying fraction of non-nulls,  \figref{saffron} reports the maximum FDP and statistical power (TDP) of CwSAFFRON along with three other procedures, SAFFRON, CwLORD++, and LORD++. We observe that SAFFRON and CwSAFFRON have FDR greater than the nominal level  of $0.1$ when the fraction of non-nulls $\pi_1$ is small (less than $0.3$), but is below the nominal level when $\pi_1$ gets larger. And the FDR of both LORD++ and CwLORD++ are generally much smaller than that of SAFFRON and CwSAFFRON. This is consistent with our discussion about SAFFRON in Section~\ref{sec:onlineFDR}. On the other hand, the power of CwSAFFRON dominates that of SAFFRON for all $\pi_1$, and is larger than that of CwLORD++ when $\pi_1$ exceeds $0.3$. As the fraction of non-nulls increases, CwSAFFRON achieves a faster increase in power than CwLORD++. A similar phenomenon can be also seen between SAFFRON and LORD++ (which was also noted by~\citet{ramdas2018saffron}).


\begin{figure}[H]
	\begin{tabular}{ll}
		\begin{subfigure}[b]{0.45\textwidth} 
			\centering 
			\includegraphics[width=\textwidth]{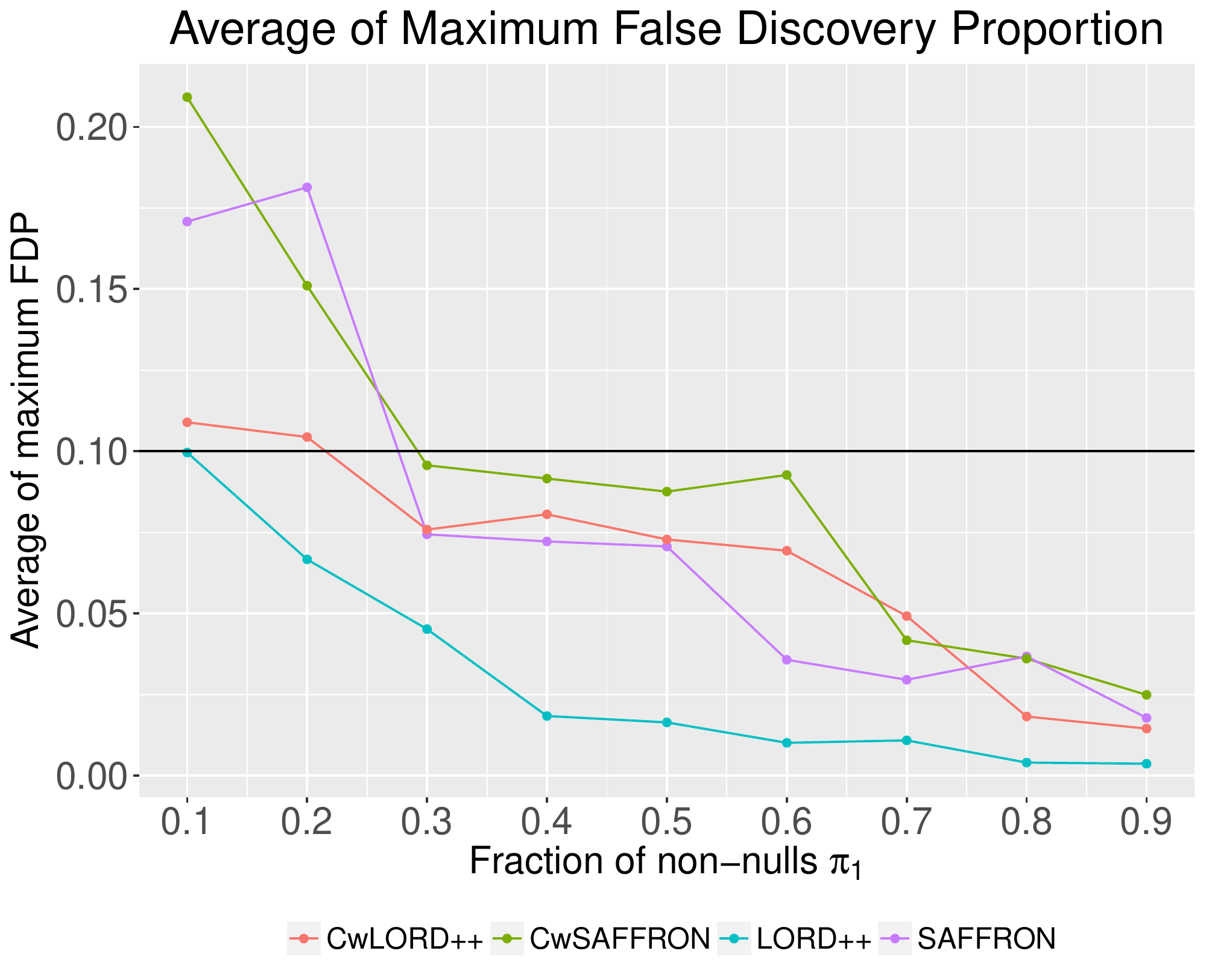} \caption{}
		\end{subfigure} & 
		\begin{subfigure}[b]{0.45\textwidth} \centering \includegraphics[width=\textwidth]{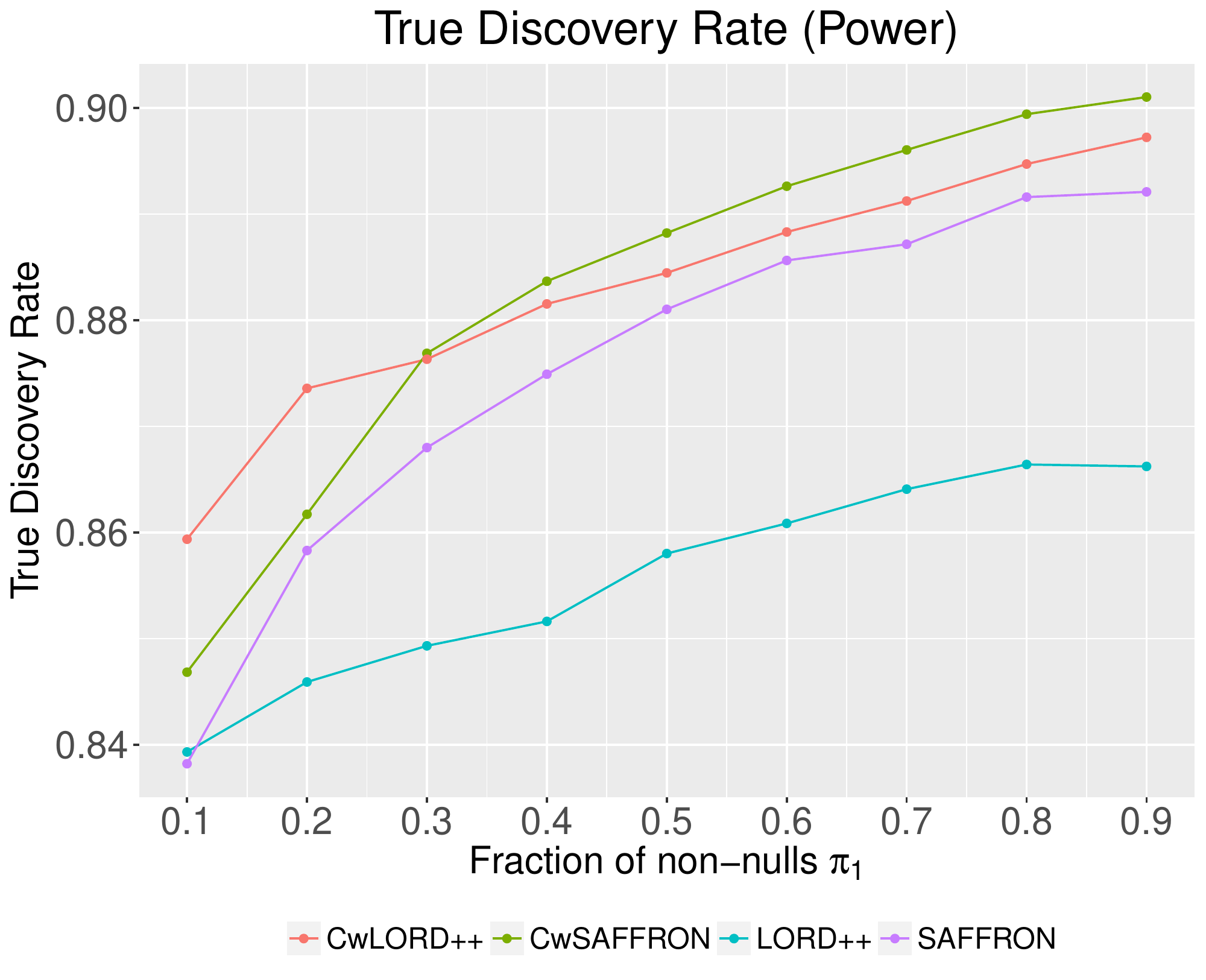}  \caption{} \end{subfigure}
	\end{tabular}
	\caption{The average of max FDP and TDR (power) for our proposed CwSAFFRON, CwLORD++, along with SAFFRON and LORD++ with varying the fraction non-nulls ($\pi_1$) under the normal means model.  The nominal FDR control level $\alpha = 0.1$. As mentioned in the description of the synthetic data experiments, the average of max FDP is an overestimate of FDR. }
	\label{fig:saffron}
	
\end{figure}

\end{document}